\documentclass[amstex,10pt,reqno]{amsart}%
\usepackage{amsmath,amsfonts,amssymb,amsthm,hyperref,enumerate,graphicx,cite}
\usepackage[table]{xcolor}
\usepackage{graphicx}
\usepackage{amsmath}
\usepackage{amsfonts}
\usepackage{amssymb}%
\providecommand{\U}[1]{\protect\rule{.1in}{.1in}}
\definecolor{anti-flashwhite}{rgb}{0.95, 0.95, 0.96}
\definecolor{gray(x11gray)}{rgb}{0.75, 0.75, 0.75}
\newtheorem{theorem}{Theorem}
\theoremstyle{plain}

\newtheorem{corollary}{Corollary}

\newtheorem{definition}{Definition}

\newtheorem{lemma}{Lemma}

\newtheorem{remark}{Remark}

\numberwithin{equation}{section}
\allowdisplaybreaks
\begin{document}

\begin{center}
{\Large \textbf{ Riemann-Liouville type fractional a new generalization of 
Bernstein-Kantorovich operators }}\newline\vspace{.9cm} 

\textbf{{Re\c{s}at
Aslan}}

\vspace{.2cm} Department of Mathematics, Van Yuzuncu Yil University, 65080
Van, Turkey\newline

E-mails: resat63@hotmail.com
\end{center}

\vspace{.8cm}

{\textbf{Abstract.}} Approximation theory is a substantial field of
mathematical analysis that emerged in the 19th century and has been developed
by mathematicians across the globe ever since. Its importance has increased
over time, as it provides solutions to numerous scientific challenges not only
in mathematics but also in fields like as physics and engineering etc. In the
present work, we construct Riemann-Liouville type fractional a new generalization of Bernstein-Kantorovich type
operators. First, we obtain the moment and central moments from some basic
calculations. Also, we study several direct and local approximation outcomes of the
constructed operators. Next, we serve up certain graphical and numerical
results to demonstrate the convergence, accuracy and significance of
constructed operators. Further, we provide bivariate version of the newly
constructed operators and establish degree of approximation through of partial
and complete modulus of continuity. Lastly, we present some graphical
representations and maximum error of approximation tables to verify the convergence
behavior of bivariate form of related operators based on various
parameters.\newline


%\maketitle

\section{Introduction}

Approximation theory has been a cornerstone of mathematical analysis for a
long time, with its influence extending across pure mathematics, computational
science, and practical domains such as engineering and computer graphics. A
pivotal and widely recognized result in this field is the Weierstrass
Approximation Theorem, which demonstrates that any continuous function defined
on a closed interval can be uniformly approximated using polynomials. In 1912,
Bernstein \cite{bernstein1912} devised a remarkable and elegant proof of
Weierstrass's approximation theorem using a sequence of polynomials.
The emergence of fractional calculus has brought a transformative perspective
to operator theory within the realm of approximation. The convergence of
fractional calculus and operator theory constitutes an expanding area of
research with applications extending beyond the realm of pure mathematics into
interdisciplinary fields such as control theory, signal processing, and
material science. Fractional derivatives and integrals, especially those of
the Riemann-Liouville type have proven to be powerful instruments for modeling
systems characterized by memory effects or hereditary behavior, including
complex fluids, mathematical biology \cite{13} to the modeling of viscoelastic
materials \cite{14}, where integer-order derivatives fail to capture the
underlying dynamics. This versatility facilitates a more accurate
correspondence between mathematical models and empirical phenomena, thereby
improving the precision and reliability of simulations and predictive analyses.

Recent developments in approximation theory have concentrated on including
shape parameters $0\leq\alpha\leq1$ and $-1\leq\lambda\leq1$, the model gains significant flexibility, into traditional operators significantly
improving their adaptability and accuracy.

In \cite{Chen}, Chen et al. invented a new class of Bernstein operators by
using the shape parameter $\alpha$ for real number $0\leq\alpha\leq1$.
Recently, Aktu\u{g}lu et al. \cite{aktuglu} presented a new generalized
blending version of $\alpha$-Bernstein operators as:
\begin{equation}
L_{m}^{(\alpha,s)}\left(  \phi;z\right)  =\sum\limits_{m=0}^{j}\phi(\frac
{m}{j})q_{m,j}^{(\alpha,s)}(z)\;\;\;\;\;\;\left(  z\in\lbrack0,1]\right)
\label{e2}%
\end{equation}
where $0\leq\alpha\leq1,s\in%
%TCIMACRO{\U{2115} }%
%BeginExpansion
\mathbb{N}
%EndExpansion
\cup\left\{  0\right\}  ,\phi\in C[0,1]$ and
\[
q_{m,j}^{(\alpha,s)}(z)=\left\{
\begin{array}
[c]{c}%
\binom{j}{m}z^{m}\left(  1-z\right)  ^{j-m},\text{ }\;\text{for }(m<s),\\
\\
\left(  1-\alpha\right)  \binom{j-s}{m-s}z^{m-s+1}(1-z)^{j-m}\\
+\left(  1-\alpha\right)  \binom{j-s}{m}z^{m}(1-z)^{j-s-m+1}\\
\text{ \ \ \ \ }+\alpha\binom{j}{m}z^{m}\left(  1-z\right)  ^{j-m},\text{ for
}(m\geq s).\text{ }\;\text{\ }%
\end{array}
\right.  \;\;\;\;\;\;\;\;\;
\]
The binomial coefficients in the above relation can be given by
\[
\binom{j}{m}=\left\{
\begin{array}
[c]{ll}%
\frac{j!}{m!(j-m)!}, & ~~~(0\leq m\leq j),\\
0, & ~~~(otherwise).
\end{array}
\right.
\]
In particular taking $s=1$ in (\ref{e2}), one can get classical Bernstein operators
\cite{bernstein1912} and $s=2$, then it turn out to the $\alpha$-Bernstein operators respectively. The
Bernstein basis function, defined as $q_{m,j}(z)=\binom{j}{m}z^{m}\left(
1-z\right)  ^{j-m}$, is a polynomial used as the foundation for constructing
B\'{e}zier curves and surfaces. These functions serve as building
blocks due to their advantageous properties, such as smoothness, locality, and
their ability to form a partition of unity. B\'{e}zier curves and
surfaces, created using Bernstein basis functions, are find extensive application in geometric modeling, computer-aided geometric design (CAGD), and approximation theory. In the last decade, many researchers have analyzed different
approximation properties of various generalized or modified linear positive
operators using the above-mentioned shape parameters in order to get better
approximation results. For instance; Mohiuddine et al. \cite{samtaaammas}
improved the Kantorovich modification of $\alpha$-Bernstein operators.
$\alpha$-Bernstein operators in terms of $q$-integers are introduced by
Agrawal et al. \cite{agrawal2022rate}. In \cite{mohiuddine1}, \"{O}zger et al.
studied a novel class of Bernstein-Schurer operators which include shape
parameter $\alpha$. Kajla et al. \cite{kajacammn} obtained some approximation
results of Durrmeyer type modification for the $\alpha$-Bernstein operators.
In \cite{usta0}, Ansari et al. examined some generalization of
$\alpha-$Szasz-Mirakyan operators. Riemann-Liouville type fractional
Kantorovich form of $\alpha$-Bernstein operators proposed and studied by
Berwal et al. \cite{berwalapproximation}. Aslan et al. \cite{aslanmursaleen}
investigated several approximation results for a new class of $\lambda-$Bernstein operators. Aral et
al. \cite{aralerbaymc} presented a new generalization of Baskakov type
operators via shape parameter $\alpha.$ Cai et al. \cite{cai-aslan2021}
proposed a new modification of $(\lambda,q)-$Bernstein operators. We also
bring to the attention of readers similar recent
studies:\cite{nasirrao,baytunc,MuhiddinFaruk,turhan2024,odabacsi2023parametric,srivastava2019,cai2022statistical,aslan1,aslan2024,ajay2020,ayman2024,nadeem2,cai-aslan2024,nadeem4,acu,ozger2021}%
.

Now, before moving forward, let us revisit some well-known definitions from the literature.

\begin{definition}
The function which is known Gamma, $\Gamma:(0,\infty)\rightarrow%
%TCIMACRO{\U{211d} }%
%BeginExpansion
\mathbb{R}
%EndExpansion
$ is given by
\[
\Gamma(\eta)=\overset{\infty}{\underset{0}{%
%TCIMACRO{\dint }%
%BeginExpansion
{\displaystyle\int}
%EndExpansion
}}e^{-v}v^{\eta-1}dv,\text{ \ \ }\operatorname{Re}(\eta)>0.
\]

\end{definition}

\begin{definition}
The Euler's integral Beta function is defined as:
\[
B(y,z)=\overset{1}{\underset{0}{%
%TCIMACRO{\dint }%
%BeginExpansion
{\displaystyle\int}
%EndExpansion
}}v^{y-1}(1-v)^{z-1}dv,\text{ \ \ }\operatorname{Re}(y),\text{ }%
\operatorname{Re}(z)>0.
\]

Additionally, for \(u, v > 0\), the relationship between the Gamma and Beta functions is given by:
\[
B(u,v)=\frac{\Gamma(u)\Gamma(v)}{\Gamma(u+v)}.
\]

\end{definition}

\begin{definition}(\cite{samko}) For the order of
$\alpha>0$ and integrable function $\phi\in\lbrack0,\infty)$, Riemann-Liouville type fractional integral is defined by
\[
(a^{+}I_{z}^{\alpha}\phi)(.)=\frac{1}{\Gamma(\alpha)}\overset{z}{\underset{a}{%
%TCIMACRO{\dint }%
%BeginExpansion
{\displaystyle\int}
%EndExpansion
}}(z-u)^{\alpha-1}\phi(u)du,\text{ \ \ }z>0.
\]

\end{definition}

In 2024, Kadak \cite{kadak} constructed for $\eta>0,$ $m\in%
%TCIMACRO{\U{2115} }%
%BeginExpansion
\mathbb{N}
%EndExpansion
$ the following fractional Bernstein-Kantorovich operators%
\[
K_{m}^{\eta}(\phi;z)=\eta\sum\limits_{j=0}^{m}q_{m,j}(z)\overset{1}%
{\underset{0}{%
%TCIMACRO{\dint }%
%BeginExpansion
{\displaystyle\int}
%EndExpansion
}}(1-t)^{\eta-1}\phi(\frac{j+t}{m+1})dt,
\]

where $z\in\lbrack0,1]$ and $q_{m,j}(z)=\binom{j}{m}z^{m}\left(  1-z\right)
^{j-m}.$

In \cite{kadak}, Kadak investigated some moment estimates, order of
convergence, graphical reprenstations and advantage of these new operators. In another work by Kadak \cite{kadak2} provides significant insights into a novel family of multivariate neural network operators, incorporating the Riemann-Liouville fractional integral operator. By deeply exploring their properties, Kadak established theoretical results that enhance the understanding of how fractional calculus can improve the performance and adaptability of neural network operators in approximating multivariate functions. Also, Baxhaku et al. \cite{behar} introduced Kantorovich-type deep neural network operators constructed using Riemann–Liouville fractional integrals.
Driven by the growing interest in operators and the pursuit of improved approximation properties, we examine a new generalization of Riemann-Liouville-type fractional Bernstein-Kantorovich operators as outlined below:%
\begin{equation}
\Re_{m,\eta,\gamma}^{(\alpha,s)}(\phi;z)=\Gamma(\eta+1)\sum\limits_{j=0}%
^{m}q_{m,j}^{(\alpha,s)}(z)\overset{1}{\underset{0}{%
%TCIMACRO{\dint }%
%BeginExpansion
{\displaystyle\int}
%EndExpansion
}}\frac{(1-t)^{\eta-1}}{\Gamma(\eta)}\phi(\frac{j+t^{\gamma}}{m+1})dt,
\label{e7}%
\end{equation}

where $\eta>0,$ $\gamma>0,$ $0\leq\alpha\leq1,$ $s\in%
%TCIMACRO{\U{2115} }%
%BeginExpansion
\mathbb{N}
%EndExpansion
\cup\left\{  0\right\}  ,$ $\phi\in C[0,1]$ and $q_{m,j}^{(\alpha,s)}(z)$ are
defined by (\ref{e2}).

The structure of this paper is arranged as follows: In Section \ref{sec2}, we give some
auxiliary results such as moments and central moments calculations. In Section
\ref{sec3}, we present uniform convergence theorem, rate of convergence by
using Peetre's $K$-functional and Lipschitz-type continuous
functions for the constructed operators. Also, we attain certain graphical results
and numerical maximum error of approximation tables for some fixed parameters to
compare the effectiveness of the operators $\Re_{m,\eta,\gamma}^{(\alpha
,s)}(\phi;z)$. In Section \ref{sec4}, we consider the bivariate extension of
operators (\ref{e7}) and derive the degree of convergence with the use of
partial and complete modulus of continuity. Lastly, we present some representations and numerical absolute error tables for the approximation in the bivariate case of the operators (\ref{e7}).

\section{Auxiliary results}

In this section, we begin our investigation by examining key estimates
concerning moments and central moments, which play a crucial role in
establishing the main results. Let us denote $e_{z}(t)=t^{z},$ where
$z=\{0,1,2\}$. \label{sec2}

\begin{lemma}
(\cite{aktuglu}) For the operators defined by (\ref{e2}), we get%
\begin{align*}
L_{m}^{(\alpha,s)}(e_{0};z)  &  =1,\\
L_{m}^{(\alpha,s)}(e_{1};z)  &  =z,\\
L_{m}^{(\alpha,s)}(e_{2};z)  &  =z^{2}+\frac{z(1-z)\left[  m+(1-\alpha
)s(s-1)\right]  }{m^{2}}.
\end{align*}

\end{lemma}

\begin{lemma}
For the operators given by (\ref{e7}), following recurrence relation%
\[
\Re_{m,\eta,\gamma}^{(\alpha,s)}(e_{i};z)=\frac{\Gamma(\eta+1)}{(m+1)^{i}}%
\sum\limits_{n=0}^{i}\binom{i}{n}m^{n}L_{m}^{(\alpha,s)}(e_{n};z)\frac
{\Gamma(\gamma(i-n)+1)}{\Gamma(\eta+\gamma(i-n)+1)},
\]

holds, where $L_{m}^{(\alpha,s)}(e_{n};z)$ given by (\ref{e2}).

\begin{proof}
In view of the definition of (\ref{e7}) and $L_{m}^{(\alpha,s)}(e_{n}%
;z)$, one can arrive%
\begin{align*}
\Re_{m,\eta,\gamma}^{(\alpha,s)}(e_{i};z)  &  =\Gamma(\eta+1)\sum
\limits_{j=0}^{m}q_{m,j}^{(\alpha,s)}(z)\overset{1}{\underset{0}{%
%TCIMACRO{\dint }%
%BeginExpansion
{\displaystyle\int}
%EndExpansion
}}\frac{(1-t)^{\eta-1}}{\Gamma(\eta)}\left(  \frac{j+t^{\gamma}}{m+1}\right)
^{i}dt\\
&  =\frac{\Gamma(\eta+1)}{\Gamma(\eta)(m+1)^{i}}\sum\limits_{j=0}^{m}%
q_{m,j}^{(\alpha,s)}(z)\overset{1}{\underset{0}{%
%TCIMACRO{\dint }%
%BeginExpansion
{\displaystyle\int}
%EndExpansion
}}(1-t)^{\eta-1}\sum\limits_{n=0}^{i}\binom{i}{n}j^{n}t^{\gamma(i-n)}dt\\
&  =\frac{\Gamma(\eta+1)}{\Gamma(\eta)(m+1)^{i}}\sum\limits_{j=0}^{m}\left(
\sum\limits_{n=0}^{i}\binom{i}{n}j^{n}B(\eta,\gamma(i-n)+1)\right)
q_{m,j}^{(\alpha,s)}(z)\\
&  =\frac{\Gamma(\eta+1)}{(m+1)^{i}}\sum\limits_{j=0}^{m}q_{m,j}^{(\alpha
,s)}(z)\sum\limits_{n=0}^{i}\binom{i}{n}j^{n}\frac{\Gamma(\gamma
(i-n)+1)}{\Gamma(\eta+\gamma(i-n)+1)}\\
&  =\frac{\Gamma(\eta+1)}{(m+1)^{i}}\sum\limits_{n=0}^{i}\binom{i}{n}%
m^{n}L_{m}^{(\alpha,s)}(e_{n};z)\frac{\Gamma(\gamma(i-n)+1)}{\Gamma
(\eta+\gamma(i-n)+1)},
\end{align*}

which ends the assertion.
\end{proof}
\end{lemma}

\begin{lemma}
\label{L2} For the operators given by (\ref{e7}), one can get following
moments:
\begin{align*}
\Re_{m,\eta,\gamma}^{(\alpha,s)}(e_{0};z)  &  =1,\\
\Re_{m,\eta,\gamma}^{(\alpha,s)}(e_{1};z)  &  =\frac{m}{m+1}z+\frac{1}%
{m+1}\frac{\eta!\gamma!}{(\eta+\gamma)!},\\
\Re_{m,\eta,\gamma}^{(\alpha,s)}(e_{2};z)  &  =\frac{m^{2}}{(m+1)^{2}}%
z^{2}+\frac{z(1-z)\left[  m+(1-\alpha)s(s-1)\right]  }{(m+1)^{2}}\\
&  +\frac{2m}{(m+1)^{2}}\frac{\eta!\gamma!}{(\eta+\gamma)!}z+\frac
{1}{(m+1)^{2}}\frac{\eta!(2\gamma)!}{(\eta+2\gamma)!}.
\end{align*}

\end{lemma}

\begin{proof}
By the definition of (\ref{e7}), it follows
\begin{align*}
\Re_{m,\eta,\gamma}^{(\alpha,s)}(e_{0};z)  &  =\Gamma(\eta+1)\sum
\limits_{j=0}^{m}q_{m,j}^{(\alpha,s)}(z)\overset{1}{\underset{0}{%
%TCIMACRO{\dint }%
%BeginExpansion
{\displaystyle\int}
%EndExpansion
}}\frac{(1-t)^{\eta-1}}{\Gamma(\eta)}dt\\
&  =\frac{\Gamma(\eta+1)}{\Gamma(\eta)}\sum\limits_{j=0}^{m}q_{m,j}%
^{(\alpha,s)}(z)B(\eta,1)\\
&  =\sum\limits_{j=0}^{m}q_{m,j}^{(\alpha,s)}(z)=1.
\end{align*}
Also,%
\begin{align*}
\Re_{m,\eta,\gamma}^{(\alpha,s)}(e_{1};z)  &  =\Gamma(\eta+1)\sum
\limits_{j=0}^{m}q_{m,j}^{(\alpha,s)}(z)\overset{1}{\underset{0}{%
%TCIMACRO{\dint }%
%BeginExpansion
{\displaystyle\int}
%EndExpansion
}}\frac{(1-t)^{\eta-1}}{\Gamma(\eta)}\left(  \frac{j+t^{\gamma}}{m+1}\right)
dt\\
&  =\frac{\Gamma(\eta+1)}{\Gamma(\eta)}\sum\limits_{j=0}^{m}q_{m,j}%
^{(\alpha,s)}(z)\frac{j}{m+1}\overset{1}{\underset{0}{%
%TCIMACRO{\dint }%
%BeginExpansion
{\displaystyle\int}
%EndExpansion
}}(1-t)^{\eta-1}dt\\
&  +\frac{\Gamma(\eta+1)}{\Gamma(\eta)}\sum\limits_{j=0}^{m}q_{m,j}%
^{(\alpha,s)}(z)\frac{1}{m+1}\overset{1}{\underset{0}{%
%TCIMACRO{\dint }%
%BeginExpansion
{\displaystyle\int}
%EndExpansion
}}(1-t)^{\eta-1}t^{\gamma}dt\\
&  =\frac{\Gamma(\eta+1)m}{\Gamma(\eta)(m+1)}\sum\limits_{j=0}^{m}%
q_{m,j}^{(\alpha,s)}(z)\frac{j}{m}B(\eta,1)\\
&  +\frac{\Gamma(\eta+1)}{\Gamma(\eta)(m+1)}\sum\limits_{j=0}^{m}%
q_{m,j}^{(\alpha,s)}(z)B(\eta,\gamma+1)\\
&  =\frac{m}{m+1}\sum\limits_{j=0}^{m}q_{m,j}^{(\alpha,s)}(z)\frac{j}{m}%
+\frac{\Gamma(\eta+1)\Gamma(\gamma+1)}{(m+1)\Gamma(\eta+\gamma+1)}%
\sum\limits_{j=0}^{m}q_{m,j}^{(\alpha,s)}(z)\\
&  =\frac{m}{m+1}z+\frac{1}{m+1}\frac{\eta!\gamma!}{(\eta+\gamma)!}.
\end{align*}
\begin{align*}
\Re_{m,\eta,\gamma}^{(\alpha,s)}(e_{2};z)  &  =\Gamma(\eta+1)\sum
\limits_{j=0}^{m}q_{m,j}^{(\alpha,s)}(z)\overset{1}{\underset{0}{%
%TCIMACRO{\dint }%
%BeginExpansion
{\displaystyle\int}
%EndExpansion
}}\frac{(1-t)^{\eta-1}}{\Gamma(\eta)}\left(  \frac{j+t^{\gamma}}{m+1}\right)
^{2}dt\\
&  =\frac{\Gamma(\eta+1)}{\Gamma(\eta)}\sum\limits_{j=0}^{m}q_{m,j}%
^{(\alpha,s)}(z)\frac{j^{2}}{(m+1)^{2}}\overset{1}{\underset{0}{%
%TCIMACRO{\dint }%
%BeginExpansion
{\displaystyle\int}
%EndExpansion
}}(1-t)^{\eta-1}dt\\
&  +\frac{2\Gamma(\eta+1)}{\Gamma(\eta)}\sum\limits_{j=0}^{m}q_{m,j}%
^{(\alpha,s)}(z)\frac{j}{(m+1)^{2}}\overset{1}{\underset{0}{%
%TCIMACRO{\dint }%
%BeginExpansion
{\displaystyle\int}
%EndExpansion
}}(1-t)^{\eta-1}t^{\gamma}dt\\
&  +\frac{\Gamma(\eta+1)}{\Gamma(\eta)}\sum\limits_{j=0}^{m}q_{m,j}%
^{(\alpha,s)}(z)\frac{1}{(m+1)^{2}}\overset{1}{\underset{0}{%
%TCIMACRO{\dint }%
%BeginExpansion
{\displaystyle\int}
%EndExpansion
}}(1-t)^{\eta-1}t^{2\gamma}dt\\
&  =\frac{\Gamma(\eta+1)m^{2}}{\Gamma(\eta)(m+1)^{2}}\sum\limits_{j=0}%
^{m}q_{m,j}^{(\alpha,s)}(z)\frac{j^{2}}{m^{2}}B(\eta,1)\\
&  +\frac{2\Gamma(\eta+1)m}{\Gamma(\eta)(m+1)^{2}}\sum\limits_{j=0}^{m}%
q_{m,j}^{(\alpha,s)}(z)\frac{j}{m}B(\eta,\gamma+1)\\
&  +\frac{\Gamma(\eta+1)}{\Gamma(\eta)(m+1)^{2}}\sum\limits_{j=0}^{m}%
q_{m,j}^{(\alpha,s)}(z)B(\eta,2\gamma+1)\\
&  =\frac{m^{2}}{(m+1)^{2}}z^{2}+\frac{z(1-z)\left[  m+(1-\alpha
)s(s-1)\right]  }{(m+1)^{2}}\\
&  +\frac{2m}{(m+1)^{2}}\frac{\eta!\gamma!}{(\eta+\gamma)!}z+\frac
{1}{(m+1)^{2}}\frac{\eta!(2\gamma)!}{(\eta+2\gamma)!}.
\end{align*}

Hence, the proof of Lemma \ref{L2} is obtained.
\end{proof}

\begin{corollary}
\label{C1} Let $\xi_{u}=(t-z)^{u},u={1,2}$, then for the operators (\ref{e7})
we get
\begin{align*}
\Re_{m,\eta,\gamma}^{(\alpha,s)}(\xi_{1};z)  &  =\frac{1}{m+1}\frac
{\eta!\gamma!}{(\eta+\gamma)!}-\frac{z}{m+1}=:\zeta_{m,\eta,\gamma}%
^{(\alpha,s)}(z),\\
\Re_{m,\eta,\gamma}^{(\alpha,s)}(\xi_{2};z)  &  =\frac{z^{2}}{(m+1)^{2}}%
+\frac{z(1-z)\left[  m+(1-\alpha)s(s-1)\right]  }{(m+1)^{2}}\\
&  -\frac{2z}{m+1}\frac{\eta!\gamma!}{(\eta+\gamma)!}+\frac{1}{(m+1)^{2}}%
\frac{\eta!(2\gamma)!}{(\eta+2\gamma)!}=:\varphi_{m,\eta,\gamma}^{(\alpha
,s)}(z).
\end{align*}

\end{corollary}

\begin{remark}
For the operators (\ref{e7}), we arrive at the following relations:

$\vartriangleright$ When $\gamma=1,$ $s=2,$ then (\ref{e7}) turn into the
operators proposed by Berwal et al. \cite{berwalapproximation}.

$\vartriangleright$ When $\eta=1$, then (\ref{e7}) turn into the operators
given by Baytun{\c{c}} et al. \cite{baytunc}.

$\vartriangleright$ When $\gamma=\eta=1,$ $s=2,$ then (\ref{e7}) turn into the
operators given by Mohiuddine et al. \cite{samtaaammas}$.$

$\vartriangleright$ When $\alpha=\eta=1,$ $s=2,$ then (\ref{e7}) turn into the
operators given by \"{O}zarslan et al. \cite{ozarslan}$.$
\end{remark}

\section{Direct and Local estimates of $\Re_{m,\eta,\gamma}^{(\alpha,s)}$}

\label{sec3}

Let the space $C[0,1]$ denote the real-valued continuous functions on $[0,1]$
and which is endowed with the sup-norm $\left\Vert \phi\right\Vert
_{C[0,1]}=\underset{z\in\lbrack0,1]}{\sup}\left\vert \phi(z)\right\vert .$

\begin{theorem}
\label{u1} For operators given by
(\ref{e7}), one get following relation
\[
\left\vert \Re_{m,\eta,\gamma}^{(\alpha,s)}(\phi;z)\right\vert \leq\left\Vert
\phi\right\Vert .
\]

\begin{proof}
From operators (\ref{e7}) and by Lemma \ref{L2}, we arrive
at
\begin{align*}
\left\vert \Re_{m,\eta,\gamma}^{(\alpha,s)}(\phi;z)\right\vert  &  =\left\vert
\Gamma(\eta+1)\sum\limits_{j=0}^{m}q_{m,j}^{(\alpha,s)}(z)\overset
{1}{\underset{0}{%
%TCIMACRO{\dint }%
%BeginExpansion
{\displaystyle\int}
%EndExpansion
}}\frac{(1-t)^{\eta-1}}{\Gamma(\eta)}\phi\left(  \frac{j+t^{\gamma}}%
{m+1}\right)  dt\right\vert \\
&  \leq\Gamma(\eta+1)\sum\limits_{j=0}^{m}q_{m,j}^{(\alpha,s)}(z)\overset
{1}{\underset{0}{%
%TCIMACRO{\dint }%
%BeginExpansion
{\displaystyle\int}
%EndExpansion
}}\frac{(1-t)^{\eta-1}}{\Gamma(\eta)}\left\vert \phi\left(  \frac{j+t^{\gamma
}}{m+1}\right)  \right\vert dt\\
&  \leq\left\Vert \phi\right\Vert \Gamma(\eta+1)\sum\limits_{j=0}^{m}%
q_{m,j}^{(\alpha,s)}(z)\overset{1}{\underset{0}{%
%TCIMACRO{\dint }%
%BeginExpansion
{\displaystyle\int}
%EndExpansion
}}\frac{(1-t)^{\eta-1}}{\Gamma(\eta)}dt\\
&  =\left\Vert \phi\right\Vert \Re_{m,\eta,\gamma}^{(\alpha,s)}(1;z)\\
&  =\left\Vert \phi\right\Vert ,
\end{align*}
which completes the assertion.
\end{proof}
\end{theorem}

\begin{theorem}
Let $\phi\in$ $C[0,1],$ then operators given by
(\ref{e7})
\[
\underset{m\rightarrow\infty}{\lim}\Re_{m,\eta,\gamma}^{(\alpha,s)}%
(\phi;z)=\phi(z),
\]
convergence uniformly on $[0,1]$.

\begin{proof}
According to the well-known Bohman-Korovkin theorem \cite{korovkin1953}, it is enough to confirm that
\[
\underset{m\rightarrow\infty}{\lim}\underset{z\in\lbrack0,1]}{\max}\left\vert
\Re_{m,\eta,\gamma}^{(\alpha,s)}(e_{l};y)-e_{l}\right\vert =0,\text{ for
}l=0,1,2.
\]
It follows immediately via Lemma \ref{L2}. Thus, desired result
is completed.
\end{proof}
\end{theorem}

Before exploring the direct and local approximation results of the operators (\ref{e7}), we first introduce some notation and revisit key definitions from the literature. Let the Peetre's K-functional is expressed as:%

\[
K_{2}(\phi,\lambda)=\underset{\kappa\in C^{2}[0,1]}{\inf}\left\{  \left\Vert
\phi-\kappa\right\Vert +\lambda\left\Vert \kappa^{\prime\prime}\right\Vert
\right\}  ,
\]

where $\lambda>0$ and $C^{2}[0,1]=\left\{  \kappa\in C[0,1]:\kappa^{\prime
},\kappa^{\prime\prime}\in C[0,1]\right\}  $.

From \cite{DeVore-1993}, there exist an absolute constant $C>0$ such that%
\begin{equation}
K_{2}(\phi;\lambda)\leq C\omega_{2}(\phi;\sqrt{\lambda}),\text{
\ \ \ \ \ \ \ }\lambda>0, \label{e10}%
\end{equation}
where $\omega_{2}(\phi;\lambda)=\underset{0<s\leq\lambda}{\sup}$
$\underset{u\in\lbrack0,1]}{\sup}\left\vert \phi(u+2s)-2\phi(u+s)+\phi
(u)\right\vert $ is the second-order modulus of smoothness of $\phi\in
C[0,1].$ Moreover, we demonstrate the usual modulus of continuity as:%
\[
\omega(\phi;\lambda):=\underset{0<\vartheta\leq\lambda}{\sup}\text{ }%
\underset{z\in\lbrack0,1]}{\sup}\left\vert \phi(u+\vartheta)-\phi
(u)\right\vert .
\]
Also, the Lipschitz function class $Lip_{M}(\kappa),$ where $M>0$ and
$0<\kappa\leq1$. If
\[
\left\vert \phi(w)-\phi(q)\right\vert \leq M\left\vert w-q\right\vert
^{\kappa},\text{ \ \ \ \ \ }(w,q\in%
%TCIMACRO{\U{211d} }%
%BeginExpansion
\mathbb{R}
%EndExpansion
).
\]

\begin{theorem}
\label{t2} Let operators $\Re_{m,\eta,\gamma}^{(\alpha,s)}(.;.)$ be given by
(\ref{e7}) and $\phi\in C[0,1],z\in\lbrack0,1]$. Then, we have%
\[
\left\vert \Re_{m,\eta,\gamma}^{(\alpha,s)}(\phi;z)-\phi(z)\right\vert
\leq2\omega(\phi;\sqrt{\varphi_{m,\eta,\gamma}^{(\alpha,s)}(z)}),
\]
where $\varphi_{m,\eta,\gamma}^{(\alpha,s)}(z)=\Re_{m,\eta,\gamma}%
^{(\alpha,s)}(\phi_{2};z)$ given by Corollary \ref{C1}.
\end{theorem}

\begin{proof}
Using the relation $\left\vert \phi(t)-\phi(z)\right\vert \leq\left(
1+\frac{\left\vert t-z\right\vert }{\delta}\right)  \omega(\phi;\delta)$ and
operating $\Re_{m,\eta,\gamma}^{(\alpha,s)}(.;z),$ we arrive
\[
\left\vert \Re_{m,\eta,\gamma}^{(\alpha,s)}(\phi;z)-\phi(z)\right\vert
\leq\left(  1+\frac{1}{\delta}\Re_{m,\eta,\gamma}^{(\alpha,s)}(\left\vert
t-z\right\vert ;z)\right)  \omega(\phi;\delta).
\]
Implementing the Cauchy-Schwarz inequality and by Corollary \ref{C1}, we
derive
\begin{align*}
\left\vert \Re_{m,\eta,\gamma}^{(\alpha,s)}(\phi;z)-\phi(z)\right\vert  &
\leq\left(  1+\frac{1}{\delta}\sqrt{\Re_{m,\eta,\gamma}^{(\alpha,s)}%
((t-z)^{2};z)}\right)  \omega(\phi;\delta)\\
&  \leq\left(  1+\frac{1}{\delta}\sqrt{\varphi_{m,\eta,\gamma}^{(\alpha
,s)}(z)}\right)  \omega(\phi;\delta).
\end{align*}
Taking $\delta=\sqrt{\varphi_{m,\eta,\gamma}^{(\alpha,s)}(z)},$ which ends
the proof.
\end{proof}

\begin{theorem}
\label{Th.3} Let $\phi\in C[0,1],z\in\lbrack0,1]$ and $\phi\in Lip_{M}%
(\kappa)$. Then, we have
\[
\left\vert \Re_{m,\eta,\gamma}^{(\alpha,s)}(\phi;z)-\phi(z)\right\vert \leq
M(\varphi_{m,\eta,\gamma}^{(\alpha,s)}(z))^{\frac{\kappa}{2}}.
\]

\begin{proof}
Using the properties of linearity and positivity of operators (\ref{e7}), we
yield
\begin{align*}
&  \left\vert \Re_{m,\eta,\gamma}^{(\alpha,s)}(\phi;z)-\phi(z)\right\vert
\leq\Re_{m,\eta,\gamma}^{(\alpha,s)}(\left\vert \phi(t)-\phi(z)\right\vert
;z)\\
&  =\Gamma(\eta+1)\sum\limits_{j=0}^{m}q_{m,j}^{(\alpha,s)}(z)\overset
{1}{\underset{0}{%
%TCIMACRO{\dint }%
%BeginExpansion
{\displaystyle\int}
%EndExpansion
}}\frac{(1-t)^{\eta-1}}{\Gamma(\eta)}\left\vert \phi\left(  \frac{j+t^{\gamma
}}{m+1}\right)  -\phi(z)\right\vert dt\\
&  \leq M\left(  \Gamma(\eta+1)\sum\limits_{j=0}^{m}q_{m,j}^{(\alpha
,s)}(z)\overset{1}{\underset{0}{%
%TCIMACRO{\dint }%
%BeginExpansion
{\displaystyle\int}
%EndExpansion
}}\frac{(1-t)^{\eta-1}}{\Gamma(\eta)}\text{ }\left\vert \frac{j+t^{\gamma}%
}{m+1}-z\right\vert ^{\kappa}dt\right)  .\text{ }%
\end{align*}
By applying Hölder's inequality with \( p_1 = \frac{2}{\kappa} \) and \( p_2 = \frac{2}{2-\kappa} \), and considering Corollary \ref{C1} and Lemma \ref{L2}, we can express it as follows:
\begin{align*}
&  \left\vert \Re_{m,\eta,\gamma}^{(\alpha,s)}(\phi;z)-\phi(z)\right\vert \\
&  \leq M\left(  \Gamma(\eta+1)\sum\limits_{j=0}^{m}q_{m,j}^{(\alpha
,s)}(z)\left(  \overset{1}{\underset{0}{%
%TCIMACRO{\dint }%
%BeginExpansion
{\displaystyle\int}
%EndExpansion
}}\frac{(1-t)^{\eta-1}}{\Gamma(\eta)}\text{ }\left\vert \frac{j+t^{\gamma}%
}{m+1}-z\right\vert ^{2}dt\right)  ^{^{\frac{\kappa}{2}}}\left(  \overset
{1}{\underset{0}{%
%TCIMACRO{\dint }%
%BeginExpansion
{\displaystyle\int}
%EndExpansion
}}\frac{(1-t)^{\eta-1}}{\Gamma(\eta)}dt\right)  ^{\frac{2-\kappa}{2}}\right)
\\
&  \leq M\left(  \Gamma(\eta+1)\sum\limits_{j=0}^{m}q_{m,j}^{(\alpha
,s)}(z)\overset{1}{\underset{0}{%
%TCIMACRO{\dint }%
%BeginExpansion
{\displaystyle\int}
%EndExpansion
}}\frac{(1-t)^{\eta-1}}{\Gamma(\eta)}\text{ }\left\vert \frac{j+t^{\gamma}%
}{m+1}-z\right\vert ^{2}dt\right)  ^{^{\frac{\kappa}{2}}}\\
&  \times\left(  \Gamma(\eta+1)\sum\limits_{j=0}^{m}q_{m,j}^{(\alpha
,s)}(z)\overset{1}{\underset{0}{%
%TCIMACRO{\dint }%
%BeginExpansion
{\displaystyle\int}
%EndExpansion
}}\frac{(1-t)^{\eta-1}}{\Gamma(\eta)}dt\right)  ^{\frac{2-\kappa}{2}}\\
&  \leq M\left(  \Re_{m,\eta,\gamma}^{(\alpha,s)}((t-z)^{2};z)\right)
^{\frac{\kappa}{2}}\left(  \Re_{m,\eta,\gamma}^{(\alpha,s)}(1;z)\right)
^{\frac{2-\kappa}{2}}=M(\varphi_{m,\eta,\gamma}^{(\alpha,s)}(z))^{\frac
{\kappa}{2}},
\end{align*}
this proves the Theorem \ref{Th.3}.
\end{proof}
\end{theorem}

\begin{theorem}
For the operators $\Re_{m,\eta,\gamma}^{(\alpha,s)},$ we obtain
\[
\left\vert \Re_{m,\eta,\gamma}^{(\alpha,s)}(\phi;z)-\phi(z)\right\vert \leq
C\omega_{2}(\phi;\frac{1}{2}\sqrt{\varphi_{m,\eta,\gamma}^{(\alpha
,s)}(z)+(\zeta_{m,\eta,\gamma}^{(\alpha,s)}(z))^{2}}+\omega(\phi;\zeta
_{m,\eta,\gamma}^{(\alpha,s)}(z)),
\]
where a constant $C>0$, $\zeta_{m,\eta,\gamma}^{(\alpha,s)}(z)$ and
$\varphi_{m,\eta,\gamma}^{(\alpha,s)}(z)$ are given by Corollary \ref{C1}.

\begin{proof}
Considering the following equation
\begin{equation}
\widehat{\Re_{m,\eta,\gamma}^{(\alpha,s)}}(\phi;z)=\phi(z)+\Re_{m,\eta,\gamma
}^{(\alpha,s)}(\phi;z)-\phi(\frac{m}{m+1}z+\frac{1}{m+1}\frac{\eta!\gamma
!}{(\eta+\gamma)!}). \label{e11}%
\end{equation}
Then, by Lemma \ref{L2}
\[
\widehat{\Re_{m,\eta,\gamma}^{(\alpha,s)}}(t-z;z)=0.
\]
In view of Taylor's formula, then
\begin{equation}
\Phi(t)=\Phi(z)+(t-z)\Phi^{\prime}(z)+\underset{z}{\overset{t}{%
%TCIMACRO{\dint }%
%BeginExpansion
{\displaystyle\int}
%EndExpansion
}}(t-u)\Phi^{\prime\prime}(u)du,\text{ \ \ }(\Phi\in C^{2}[0,1]). \label{e12}%
\end{equation}
Operating $\widehat{\Re_{m,\eta,\gamma}^{(\alpha,s)}}(.;z)$ to (\ref{e12}),
yields
\begin{align*}
&  \widehat{\Re_{m,\eta,\gamma}^{(\alpha,s)}}(\Phi;z)-\Phi(z)=\widehat
{\Re_{m,\eta,\gamma}^{(\alpha,s)}}((t-z)\Phi^{\prime}(z);z)+\widehat
{\Re_{m,\eta,\gamma}^{(\alpha,s)}}(\underset{z}{\overset{t}{%
%TCIMACRO{\dint }%
%BeginExpansion
{\displaystyle\int}
%EndExpansion
}}(t-u)\Phi^{\prime\prime}(u)du;z)\\
&  =\Phi^{\prime}(z)\widehat{\Re_{m,\eta,\gamma}^{(\alpha,s)}}(t-z;z)+\widehat
{\Re_{m,\eta,\gamma}^{(\alpha,s)}}(\underset{z}{\overset{t}{%
%TCIMACRO{\dint }%
%BeginExpansion
{\displaystyle\int}
%EndExpansion
}}(t-u)\Phi^{\prime\prime}(u)du;z)\\
&  -\underset{z}{\overset{\frac{m}{m+1}z+\frac{1}{m+1}\frac{\eta!\gamma
!}{(\eta+\gamma)!}}{%
%TCIMACRO{\dint }%
%BeginExpansion
{\displaystyle\int}
%EndExpansion
}}(\frac{m}{m+1}z+\frac{1}{m+1}\frac{\eta!\gamma!}{(\eta+\gamma)!}%
-u)\Phi^{\prime\prime}(u)du\\
&  =\Re_{m,\eta,\gamma}^{(\alpha,s)}(\underset{z}{\overset{t}{%
%TCIMACRO{\dint }%
%BeginExpansion
{\displaystyle\int}
%EndExpansion
}}(t-u)\Phi^{\prime\prime}(u)du;z)-\underset{z}{\overset{\frac{m}{m+1}%
z+\frac{1}{m+1}\frac{\eta!\gamma!}{(\eta+\gamma)!}}{%
%TCIMACRO{\dint }%
%BeginExpansion
{\displaystyle\int}
%EndExpansion
}}(\frac{m}{m+1}z+\frac{1}{m+1}\frac{\eta!\gamma!}{(\eta+\gamma)!}%
-u)\Phi^{\prime\prime}(u)du.
\end{align*}
Again using Lemma \ref{L2},
\begin{align*}
&  \left\vert \widehat{\Re_{m,\eta,\gamma}^{(\alpha,s)}}(\Phi;z)-\Phi
(z)\right\vert \\
&  \leq\left\vert \Re_{m,\eta,\gamma}^{(\alpha,s)}(\underset{z}{\overset{t}{%
%TCIMACRO{\dint }%
%BeginExpansion
{\displaystyle\int}
%EndExpansion
}}(t-u)\Phi^{\prime\prime}(u)du;z)\right\vert +\left\vert \underset
{z}{\overset{\frac{m}{m+1}z+\frac{1}{m+1}\frac{\eta!\gamma!}{(\eta+\gamma)!}}{%
%TCIMACRO{\dint }%
%BeginExpansion
{\displaystyle\int}
%EndExpansion
}}(\frac{m}{m+1}z+\frac{1}{m+1}\frac{\eta!\gamma!}{(\eta+\gamma)!}%
-u)\Phi^{\prime\prime}(u)du\right\vert \\
&  \leq\Re_{m,\eta,\gamma}^{(\alpha,s)}(\left\vert \underset{z}{\overset{t}{%
%TCIMACRO{\dint }%
%BeginExpansion
{\displaystyle\int}
%EndExpansion
}}(t-u)\right\vert \left\vert \Phi^{\prime\prime}(u)\right\vert
du;z)+\underset{z}{\overset{\frac{m}{m+1}z+\frac{1}{m+1}\frac{\eta!\gamma
!}{(\eta+\gamma)!}}{%
%TCIMACRO{\dint }%
%BeginExpansion
{\displaystyle\int}
%EndExpansion
}}\left\vert \frac{m}{m+1}z+\frac{1}{m+1}\frac{\eta!\gamma!}{(\eta+\gamma
)!}-u\right\vert \left\vert \Phi^{\prime\prime}(u)\right\vert du\\
&  \leq\left\Vert \Phi^{\prime\prime}\right\Vert \left\{  \Re_{m,\eta,\gamma
}^{(\alpha,s)}((t-z)^{2};z)+\left(  \frac{m}{m+1}z+\frac{1}{m+1}\frac
{\eta!\gamma!}{(\eta+\gamma)!}-z\right)  ^{2}\right\} \\
&  =\left\Vert \Phi^{\prime\prime}\right\Vert \left\{  \varphi_{m,\eta,\gamma
}^{(\alpha,s)}(z)+(\zeta_{m,\eta,\gamma}^{(\alpha,s)}(z))^{2}\right\}  .
\end{align*}
Considering (\ref{e7}) and (\ref{e11}), hence
\begin{align}
\left\vert \widehat{\Re_{m,\eta,\gamma}^{(\alpha,s)}}(\phi;z)\right\vert  &
\leq2\left\Vert \phi\right\Vert +\left\vert \Re_{m,\eta,\gamma}^{(\alpha
,s)}(\phi;z)\right\vert \nonumber\\
&  \leq2\left\Vert \phi\right\Vert +\left\Vert \phi\right\Vert \Re
_{m,\eta,\gamma}^{(\alpha,s)}(e_{0};z)\nonumber\\
&  \leq3\left\Vert \phi\right\Vert . \label{e13}%
\end{align}
Next, by using Theorem \ref{u1}, relations (\ref{e11}) and
(\ref{e13}) we find
\begin{align*}
\left\vert \Re_{m,\eta,\gamma}^{(\alpha,s)}(\phi;z)-\phi(z)\right\vert  &
\leq\left\vert \widehat{\Re_{m,\eta,\gamma}^{(\alpha,s)}}(\phi-\Phi
;z)-(\phi-\Phi)(z)\right\vert \\
&  +\left\vert \widehat{\Re_{m,\eta,\gamma}^{(\alpha,s)}}(\Phi;z)-\Phi
(z)\right\vert +\left\vert \phi(z)-\phi(\frac{m}{m+1}z+\frac{1}{m+1}\frac
{\eta!\gamma!}{(\eta+\gamma)!})\right\vert \\
&  \leq4\left\Vert \phi-\Phi\right\Vert +\left\{  \varphi_{m,\eta,\gamma
}^{(\alpha,s)}(z)+(\zeta_{m,\eta,\gamma}^{(\alpha,s)}(z))^{2}\right\}
\left\Vert \Phi^{\prime\prime}\right\Vert +\omega\left(  \phi;\zeta
_{m,\eta,\gamma}^{(\alpha,s)}(z)\right)  .
\end{align*}
If we take infimum over all functions $\Phi\in C^{2}[0,1]$ and by inequality
(\ref{e10}), hence
\begin{align*}
\left\vert \Re_{m,\eta,\gamma}^{(\alpha,s)}(\phi;z)-\phi(z)\right\vert  &
\leq4K_{2}(\phi;\frac{\left\{  \varphi_{m,\eta,\gamma}^{(\alpha,s)}%
(z)+(\zeta_{m,\eta,\gamma}^{(\alpha,s)}(z))^{2}\right\}  }{4})+\omega
(\phi;\zeta_{m,\eta,\gamma}^{(\alpha,s)}(z))\\
&  \leq C\omega_{2}(\phi;\frac{1}{2}\sqrt{\varphi_{m,\eta,\gamma}^{(\alpha
,s)}(z)+(\zeta_{m,\eta,\gamma}^{(\alpha,s)}(z))^{2}})+\omega(\phi
;\zeta_{m,\eta,\gamma}^{(\alpha,s)}(z)),
\end{align*}
which ends the assertion.
\end{proof}
\end{theorem}

\section{Graphical and Numerical Evaluations}

\label{sec7} In this section, Maple software is utilized to examine the
convergence properties, efficiency, and consistency of the proposed operators
$\Re_{m,\eta,\gamma}^{(\alpha,s)}$. The analysis is supported by graphical
illustrations and tables presenting approximation errors for specific fixed parameters.

 \begin{figure}[h]
\caption{Convergence of $\Re_{m,3,3}^{(0.9,4)}(\phi;z)$ to $\phi(z)=z   (z-\frac{4}{7})  sin(\pi z) $ for $m=20, 30,
70$}%
\label{fig1}%
\centering
\includegraphics[width=13cm, height=7cm]{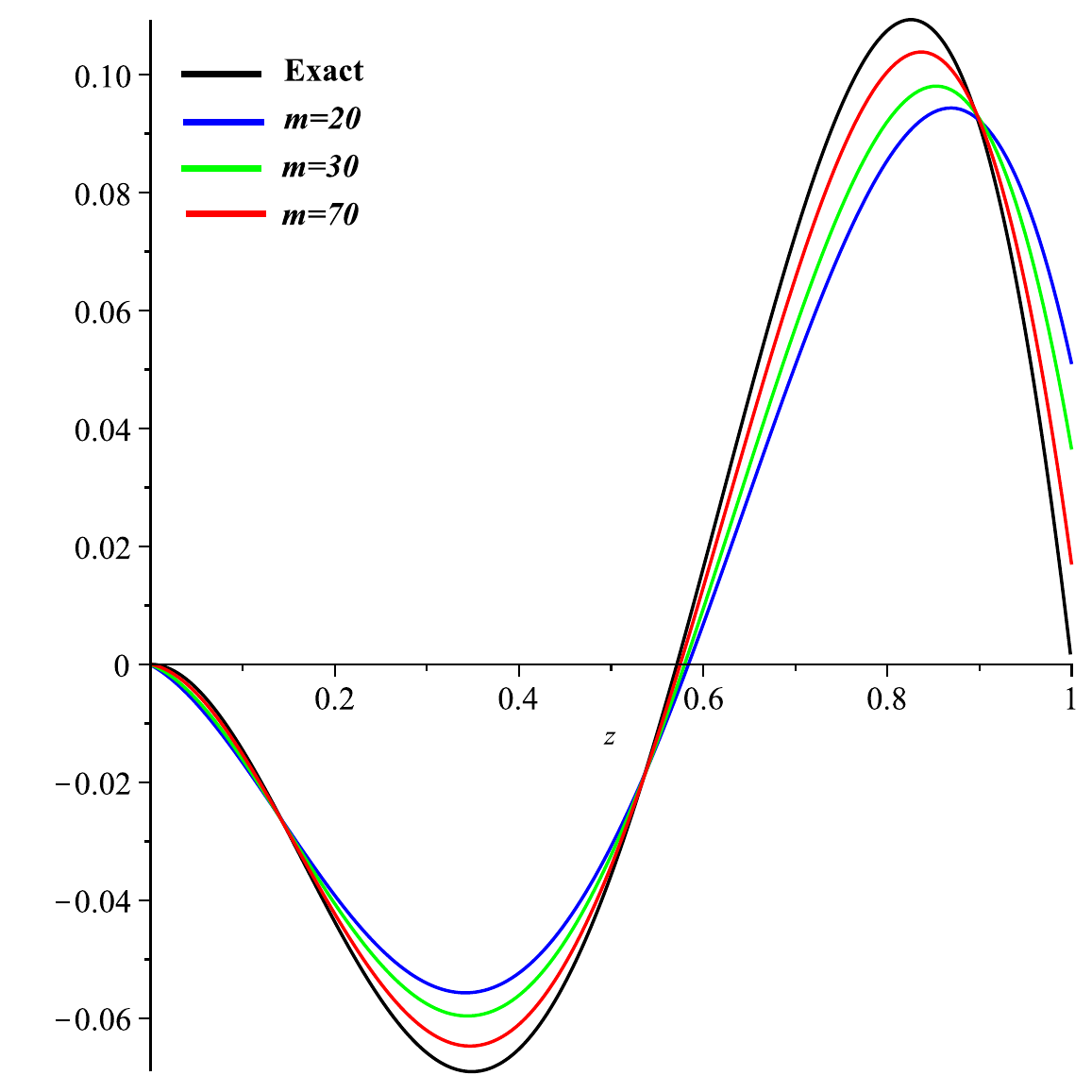}\end{figure}

\begin{table}[h]
	\caption{The absolute error of convergence of operators $\Re_{m,2,4}^{(0.9,3)}(\phi;z)$ to $\phi(z)=z   (z-\frac{4}{7})  sin(\pi z) $ for $m=40, 100, 250$}%
	\label{table1}%
	\rowcolors{0.5}{anti-flashwhite}{gray(x11gray)}
	\begin{tabular}
		[c]{c|c|c|c|c}\hline
		$z$ & $|\Re_{40,2,4}^{(0.9,3)}(\phi;z)-\phi(z)|$ & $|\Re_{100,2,4}^{(0.9,3)}(\phi;z)-\phi(z)|$ &
		$|\Re_{250,2,4}^{(0.9,3)}(\phi;z)-\phi(z)|$ & \\\hline\hline
		0.1 & 0.00123805 & 0.00053984 & 0.00022331 & \\
		0.2 & 0.00233038 & 0.00095656 & 0.00038679 & \\
		0.3 & 0.00639784 & 0.00269305 & 0.00110004 & \\
		0.4 & 0.00702828 & 0.00297029 & 0.00121524 & \\
		0.5 & 0.00255941 & 0.00105894 & 0.00042942 & \\
		0.6 & 0.00540837 & 0.00234360 & 0.00096844 & \\
		0.7 & 0.01226355 & 0.00522961 & 0.00214687 & \\
		0.8 & 0.01186536 & 0.00496537 & 0.00202220 & \\
		0.9 & 0.00107304 & 0.00062238 & 0.00028315 & \\\hline
	\end{tabular}
\end{table}

In Figure
\ref{fig1} and Table \ref{table1}, we
consider the function
\[
\phi(z)=z   (z-\frac{4}{7})  sin(\pi z)
\]
on the interval $[0,1]$. In Figure
\ref{fig1}, we demonstrate convergence of operators $\Re_{m,\eta,\gamma}^{(\alpha,s)}(\phi;z)$ to $\phi(z)=z   (z-\frac{4}{7})  sin(\pi z) $ for fixed parameters $\eta=3, \gamma=3, s=4, \alpha=0.9,$ while $m=20,30,70$ values increase  respectively. In Table \ref{table1}, we compare maximum error of approximation of operators $\Re_{m,\eta,\gamma}^{(\alpha,s)}(\phi;z)$ to $\phi(z)=z   (z-\frac{4}{7})  sin(\pi z) $ for fixed parameters $\eta=2, \gamma=4, s=3, \alpha=0.9,$ and different  values of $z\in\lbrack0,1]$ while $m=40,100,250$ values increase,  respectively. As can be seen by Figure
\ref{fig1} and Table \ref{table1} that the convergence and error of approximation of operators $\Re_{m,\eta,\gamma}^{(\alpha,s)}(\phi;z)$ turn into better in terms of some selected fixed parameters while $m$ values increase.

\begin{figure}[h]
\caption{Convergence of $\Re_{10,2,3}^{(\alpha,4)}(\phi;z)$ to $\phi(z)=(1-z)  cos(2\pi z)$
for $\alpha=0.35,0.65,0.95$}%
\label{fig2}%
\centering
\includegraphics[width=13cm, height=8cm]{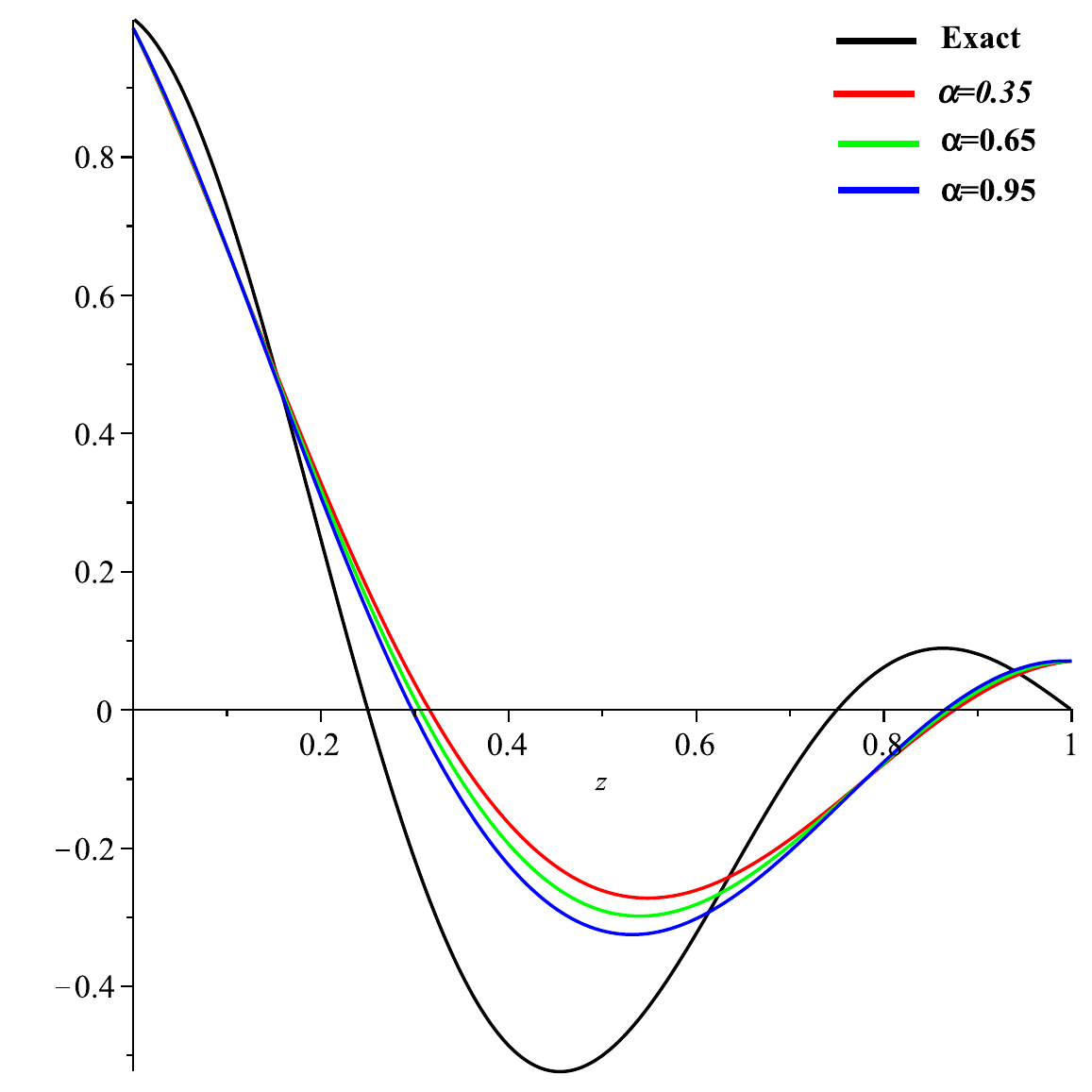}\end{figure}

\begin{table}[h]
	\caption{The absolute error of convergence of operators $\Re_{90,3,2}^{(\alpha,3)}(\phi;z)$ to $\phi(z)=(1-z)  cos(2\pi z)$
		for $\alpha=0.35,0.65,0.95$}%
	\label{table2}%
	\rowcolors{0.5}{anti-flashwhite}{gray(x11gray)}
	\begin{tabular}
		[c]{c|c|c|c|c}\hline
		$z$ & $|\Re_{90,3,2}^{(0.35,3)}(\phi;z)-\phi(z)|$ & $|\Re_{90,3,2}^{(0.65,3)}(\phi;z)-\phi(z)|$ &
		$|\Re_{90,3,2}^{(0.95,3)}(\phi;z)-\phi(z)|$ & \\\hline\hline
		0.1 & 0.01047627 & 0.01029560 & 0.01011492 & \\
		0.2 & 0.00762564 & 0.00756724 & 0.00750883 & \\
		0.3 & 0.03201223 & 0.03156654 & 0.03112084 & \\
		0.4 & 0.03975176 & 0.03910276 & 0.03845375 & \\
		0.5 & 0.02372940 & 0.02320796 & 0.02268652 & \\
		0.6 & 0.00379836 & 0.00377262 & 0.00374689 & \\
		0.7 & 0.02231546 & 0.02217607 & 0.02203669 & \\
		0.8 & 0.01974499 & 0.01951551 & 0.01928603 & \\
		0.9 & 0.00233408 & 0.00221982 & 0.00210556 & \\\hline
	\end{tabular}
\end{table}

In Figure
\ref{fig2} and Table \ref{table2}, we
examine the function
\[
\phi(z)=(1-z)  cos(2\pi z)
\]
on the interval $[0,1]$. In Figure
\ref{fig1}, we demonstrate the convergence of operators $\Re_{m,\eta,\gamma}^{(\alpha,s)}(\phi;z)$ to $\phi(z)=(1-z)  cos(2\pi z)$ for fixed parameters $\eta=2, \gamma=3, s=4, m=10,$ while $\alpha=0.35,0.65,0.95$ values increase  respectively. In Table \ref{table2}, we compare absolute error of convergence of operators $\Re_{m,\eta,\gamma}^{(\alpha,s)}(\phi;z)$ to $\phi(z)=(1-z)  cos(2\pi z) $ for fixed parameters $\eta=3, \gamma=2, s=3, m=90,$ while $\alpha=0.35,0.65,0.95$ values increase, respectively. It
can be easily seen by Figure
\ref{fig1} and Table \ref{table1} that as the $~\alpha$ values approaches 1, then the convergence of proposed operators becomes better and absolute
error decreases clearly. In other words, the operators $\Re_{m,\eta,\gamma}^{(\alpha,s)}$ offers a better
approximation result.

\begin{figure}[h]
	\caption{Convergence of $\Re_{10,3,2}^{(0.75,s)}(\phi;z)$ to $\phi(z)=22z(z-0.9) (z-0.3)$ for $s=2,5,8$ }%
	\label{fig3}%
	\centering
	\includegraphics[width=13cm, height=9cm]{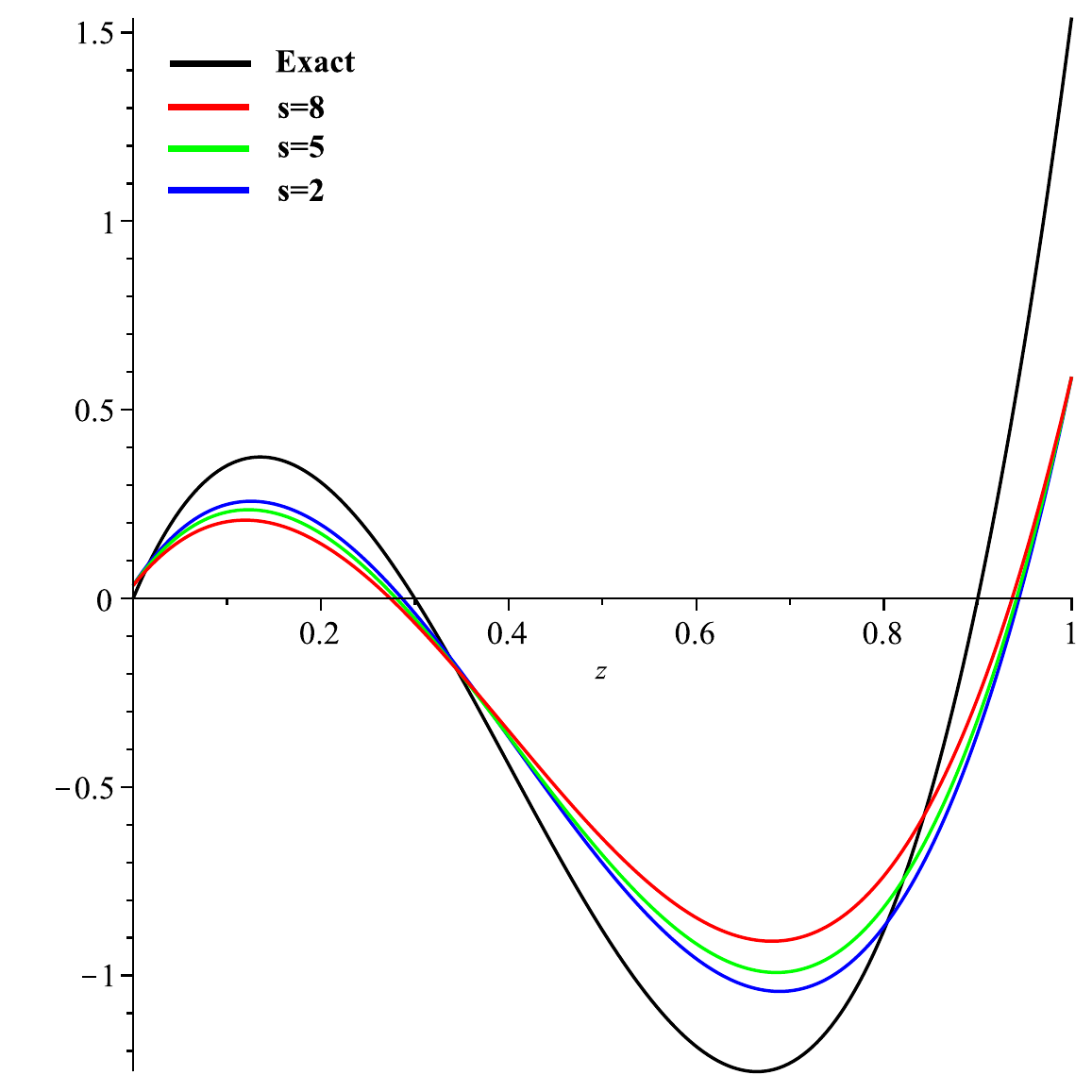}\end{figure}
\begin{table}[h]
	\caption{The absolute error of convergence of operators $\Re_{70,3,2}^{(0.75,s)}(\phi;z)$ to $\phi(z)=22z(z-0.9) (z-0.3)$ for  $s=2,5,8$ }%
	\label{table3}%
	\rowcolors{0.5}{anti-flashwhite}{gray(x11gray)}
	\begin{tabular}
		[c]{c|c|c|c|c}\hline
		$z$ & $|\Re_{70,3,2}^{(0.75,8)}(\phi;z)-\phi(z)|$ & $|\Re_{70,3,2}^{(0.75,5)}(\phi;z)-\phi(z)|$ &
		$|\Re_{70,3,2}^{(0.75,2)}(\phi;z)-\phi(z)|$ & \\\hline\hline
		0.1 & 0.02889929 & 0.02614264 & 0.02467802 & \\
		0.2 & 0.03157043 & 0.02833968 & 0.02660924 & \\
		0.3 & 0.01165659 & 0.00960805 & 0.00848310 & \\
		0.4 & 0.01898178 & 0.01881803 & 0.01879369 & \\
		0.5 & 0.04848423 & 0.04570435 & 0.04431441 & \\
		0.6 & 0.06499030 & 0.05981668 & 0.05717234 & \\
		0.7 & 0.05663953 & 0.04992080 & 0.04646074 & \\
		0.8 & 0.01157147 & 0.00478247 & 0.00127291 & \\
		0.9 & 0.08907434 & 0.08683250 & 0.08229788 & \\\hline
	\end{tabular}
\end{table}
In Figure
\ref{fig3} and Table \ref{table3}, we
consider the function
\[
\phi(z)=22z(z-0.9) (z-0.3)
\]
on the interval $[0,1]$. In Figure
\ref{fig3}, we present convergence of operators $\Re_{m,\eta,\gamma}^{(\alpha,s)}(\phi;z)$ to $\phi(z)=22z(z-0.9) (z-0.3)$ for fixed parameters $\eta=3, \gamma=2, \alpha=0.75, m=10,$ while $s=2,5,8$ values increase respectively. In Table \ref{table3}, we compare the maximum absolute error of approximation of operators $\Re_{m,\eta,\gamma}^{(\alpha,s)}(\phi;z)$ to $\phi(z)=22z(z-0.9) (z-0.3) $ for fixed parameters $\eta=3, \gamma=2, \alpha=0.75, m=70,$  while $s=2,5,8$ values increase, respectively. As shown in Figure \ref{fig3} and the data in Table \ref{table1}, the proposed operators \( \Re_{m,\eta,\gamma}^{(\alpha,s)}(\phi;z) \) demonstrate a clear improvement in approximation accuracy as the value of \( s \) decreases.

\begin{table}[h]
\caption{The comparing of error of approximation of operators $RLBK$, $BBK$, $FBK$, and $RLGBK$ to $\phi(z)=z   (z-\frac{2}{5})  (z-\frac{7}{8}) $ for $m=10,20,40,80$}%
\label{table4}%
\rowcolors{0.5}{anti-flashwhite}{gray(x11gray)}
\begin{tabular}
[c]{c|c|c|c|c|c}\hline
$m$ & $RLBK$ & $BBK$ & $FBK$ & $RLGBK$ &
\\\hline\hline
10 & 0.00871903 & 0.00888781 & 0.00921728 & 0.00854246 &\\
20 & 0.00495655 & 0.00500511 & 0.00524502 & 0.00470220 &\\
40 & 0.00264793 & 0.00266098 & 0.00280300 & 0.00247096 &\\
80 & 0.00136927 & 0.00137266 & 0.00144966 & 0.00126713 &\\
\hline
\end{tabular}
\end{table}
Moreover, in
Table \ref{table4}, we
examine the function
\[
\phi(z)=z   (z-\frac{2}{5})  (z-\frac{7}{8}) 
\]
on the interval $[0,1]$. We choose $\alpha=0.9, \eta=2, \gamma=3, s=2,$ while $m=10,20,40,80$ values increase, respectively. In order to show the advantage of proposed operators $\Re_{m,\eta,\gamma}^{(\alpha,s)}(\phi;z)$(RLGBK), we compare the maximum absolute error of approximation of operators $\Re_{m,\eta,\gamma}^{(\alpha,s)}(\phi;z)$(RLGBK) with some known operators in the literature as Berwal et al. \cite{berwalapproximation}(RLBK), Baytun{\c{c}} et al. \cite{baytunc}(BBK), Mohiuddine et al. \cite{samtaaammas}(FBK). Table \ref{table4} demonstrates that the RLGBK operators provide a significantly better approximation for $\phi(z)=z   (z-\frac{2}{5})  (z-\frac{7}{8}) $ compared to the operators such as RLBK, BBK, and FBK. 

\section{Bivariate extension of operators of $\Re_{m,\eta,\gamma}%
^{(\alpha,s)}$}

\label{sec4}

Let $I^{2}=[0,1]\times\lbrack0,1]$, $(z,y)\in I^{2}$ and $C(I^{2})$ be the
element of all continuous functions on $I^{2}$ which has the sup-norm
$\left\Vert \phi\right\Vert _{C(I^{2})}=\sup_{(z,y)\in I^{2}}\left\vert
\phi(z,y)\right\vert .$ For all $\phi\in C(I^{2}),$ $m_{1},m_{2}\in%
%TCIMACRO{\U{2115} }%
%BeginExpansion
\mathbb{N}
%EndExpansion
,(z,y)\in I^{2}$, $\eta_{1},\eta_{2}>0,$ $\gamma_{1},\gamma_{2}>0,$
$0\leq\alpha_{1,}\alpha_{2}\leq1,$ $s_{1},s_{2}\in%
%TCIMACRO{\U{2115} }%
%BeginExpansion
\mathbb{N}
%EndExpansion
\cup\left\{  0\right\}  ,$ and $q_{m_{1},j_{1}}^{(\alpha_{1},s_{1})}(z),$
$q_{m_{2},j_{2}}^{(\alpha_{2},s_{2})}(y)$ are given as in (\ref{e2})$,$ then
we get bivariate type of operators given by (\ref{e7}) as follows:
\begin{align}
\Re_{m_{1},m_{2},\eta_{1},\eta_{2},\gamma_{1},\gamma_{2}}^{(\alpha_{1}%
,\alpha_{2,},s_{1},s_{2})}(\phi;z,y)  &  =\Gamma(\eta_{1}+1)\Gamma(\eta
_{2}+1)\sum\limits_{j_{1}=0}^{m_{1}}\sum\limits_{j_{2}=0}^{m_{2}}%
q_{m_{1},j_{1}}^{(\alpha_{1},s_{1})}(z)q_{m_{2},j_{2}}^{(\alpha_{2},s_{2}%
)}(y)\nonumber\\
&  .\overset{1}{\underset{0}{%
%TCIMACRO{\dint }%
%BeginExpansion
{\displaystyle\int}
%EndExpansion
}}\overset{1}{\underset{0}{%
%TCIMACRO{\dint }%
%BeginExpansion
{\displaystyle\int}
%EndExpansion
}}\frac{(1-t_{1})^{\eta_{1}-1}}{\Gamma(\eta_{1})}\frac{(1-t_{2})^{\eta_{2}-1}%
}{\Gamma(\eta_{2})}\phi(\frac{j_{1}+t_{1}^{\gamma_{1}}}{m_{1}+1},\frac
{j_{2}+t_{2}^{\gamma_{2}}}{m_{2}+1})dt_{1}dt_{2}. \label{e15}%
\end{align}

The following Lemma follows directly from Lemma \ref{L2}.

\begin{lemma}
\label{L3} Let $e_{u,v}(z,y)=z^{u}y^{v},$ where $0\leq u+v\leq2.$ Then, for
the bivariate operators given by (\ref{e15}) we get

$\Re_{m_{1},m_{2},\eta_{1},\eta_{2},\gamma_{1},\gamma_{2}}^{(\alpha_{1}%
,\alpha_{2,},s_{1},s_{2})}(e_{0,0};z,y)=1,$

$\Re_{m_{1},m_{2},\eta_{1},\eta_{2},\gamma_{1},\gamma_{2}}^{(\alpha_{1}%
,\alpha_{2,},s_{1},s_{2})}(e_{1,0};z,y)=\frac{m_{1}}{m_{1}+1}z+\frac{1}%
{m_{1}+1}\frac{\eta_{1}!\gamma_{1}!}{(\eta_{1}+\gamma_{1})!},$

$\Re_{m_{1},m_{2},\eta_{1},\eta_{2},\gamma_{1},\gamma_{2}}^{(\alpha_{1}%
,\alpha_{2,},s_{1},s_{2})}(e_{0,1};z,y)=\frac{m_{2}}{m_{2}+1}y+\frac{1}%
{m_{2}+1}\frac{\eta_{2}!\gamma_{2}!}{(\eta_{2}+\gamma_{2})!},$

$\Re_{m_{1},m_{2},\eta_{1},\eta_{2},\gamma_{1},\gamma_{2}}^{(\alpha_{1}%
,\alpha_{2,},s_{1},s_{2})}(e_{1,1};z,y)=\left(  \frac{m_{1}}{m_{1}+1}%
z+\frac{1}{m_{1}+1}\frac{\eta_{1}!\gamma_{1}!}{(\eta_{1}+\gamma_{1})!}\right)
$.$(\frac{m_{2}}{m_{2}+1}y+\frac{1}{m_{2}+1}\frac{\eta_{2}!\gamma_{2}!}%
{(\eta_{2}+\gamma_{2})!}),$

$\Re_{m_{1},m_{2},\eta_{1},\eta_{2},\gamma_{1},\gamma_{2}}^{(\alpha_{1}%
,\alpha_{2,},s_{1},s_{2})}(e_{2,0};z,y)=\frac{m_{1}^{2}}{(m_{1}+1)^{2}}%
z^{2}+\frac{z(1-z)\left[  m_{1}+(1-\alpha_{1})s_{1}(s_{1}-1)\right]  }%
{(m_{1}+1)^{2}}$

$+\frac{2m_{1}}{(m_{1}+1)^{2}}\frac{\eta_{1}!\gamma_{1}!}{(\eta_{1}+\gamma
_{1})!}z+\frac{1}{(m_{1}+1)^{2}}\frac{\eta_{1}!(2\gamma_{1})!}{(\eta
_{1}+2\gamma_{1})!},$

$\Re_{m_{1},m_{2},\eta_{1},\eta_{2},\gamma_{1},\gamma_{2}}^{(\alpha_{1}%
,\alpha_{2,},s_{1},s_{2})}(e_{0,2};z,y)=\frac{m_{2}^{2}}{(m_{2}+1)^{2}}%
y^{2}+\frac{y(1-y)\left[  m_{2}+(1-\alpha_{2})s_{2}(s_{2}-1)\right]  }%
{(m_{2}+1)^{2}}$

$+\frac{2m_{2}}{(m_{2}+1)^{2}}\frac{\eta_{2}!\gamma_{2}!}{(\eta_{2}+\gamma
_{2})!}y+\frac{1}{(m_{2}+1)^{2}}\frac{\eta_{2}!(2\gamma_{2})!}{(\eta
_{2}+2\gamma_{2})!}.$

\begin{proof}
By the definitions of operators (\ref{e15}) and (\ref{e7}), and based on Lemma \ref{L2}, we obtain

$\Re_{m_{1},m_{2},\eta_{1},\eta_{2},\gamma_{1},\gamma_{2}}^{(\alpha_{1}%
,\alpha_{2,},s_{1},s_{2})}(e_{0,0};z,y)=\Re_{m_{1},\eta_{1},\gamma_{1}%
}^{(\alpha_{1},s_{1})}(e_{0};z)\Re_{m_{2},\eta_{2},\gamma_{2}}^{(\alpha
_{2,},s_{2})}(e_{0};y),$

$\Re_{m_{1},m_{2},\eta_{1},\eta_{2},\gamma_{1},\gamma_{2}}^{(\alpha_{1}%
,\alpha_{2,},s_{1},s_{2})}(e_{1,0};z,y)=\Re_{m_{1},\eta_{1},\gamma_{1}%
}^{(\alpha_{1},s_{1})}(e_{1};z)\Re_{m_{2},\eta_{2},\gamma_{2}}^{(\alpha
_{2,},s_{2})}(e_{0};y),$

$\Re_{m_{1},m_{2},\eta_{1},\eta_{2},\gamma_{1},\gamma_{2}}^{(\alpha_{1}%
,\alpha_{2,},s_{1},s_{2})}(e_{0,1};z,y)=\Re_{m_{1},\eta_{1},\gamma_{1}%
}^{(\alpha_{1},s_{1})}(e_{0};z)\Re_{m_{2},\eta_{2},\gamma_{2}}^{(\alpha
_{2,},s_{2})}(e_{1};y),$

$\Re_{m_{1},m_{2},\eta_{1},\eta_{2},\gamma_{1},\gamma_{2}}^{(\alpha_{1}%
,\alpha_{2,},s_{1},s_{2})}(e_{1,1};z,y)=\Re_{m_{1},\eta_{1},\gamma_{1}%
}^{(\alpha_{1},s_{1})}(e_{1};z)\Re_{m_{2},\eta_{2},\gamma_{2}}^{(\alpha
_{2,},s_{2})}(e_{1};y),$

$\Re_{m_{1},m_{2},\eta_{1},\eta_{2},\gamma_{1},\gamma_{2}}^{(\alpha_{1}%
,\alpha_{2,},s_{1},s_{2})}(e_{2,0};z,y)=\Re_{m_{1},\eta_{1},\gamma_{1}%
}^{(\alpha_{1},s_{1})}(e_{2};z)\Re_{m_{2},\eta_{2},\gamma_{2}}^{(\alpha
_{2,},s_{2})}(e_{0};y),$

$\Re_{m_{1},m_{2},\eta_{1},\eta_{2},\gamma_{1},\gamma_{2}}^{(\alpha_{1}%
,\alpha_{2,},s_{1},s_{2})}(e_{0,2};z,y)=\Re_{m_{1},\eta_{1},\gamma_{1}%
}^{(\alpha_{1},s_{1})}(e_{0};z)\Re_{m_{2},\eta_{2},\gamma_{2}}^{(\alpha
_{2,},s_{2})}(e_{2};y).$

This immediately concludes the proof of Lemma \ref{L3}.
\end{proof}
\end{lemma}

\begin{theorem}
Let the operators \( \Re_{m_{1},m_{2},\eta_{1},\eta_{2},\gamma_{1},\gamma_{2}}^{(\alpha_{1},\alpha_{2},s_{1},s_{2})}(\phi;z,y) \) be defined by (\ref{e15}). Then, for all \( \phi \in C(I^2) \), the operators (\ref{e15}) converge uniformly to \( \phi \) on \( I^2 \).
\end{theorem}

\begin{proof}
It suffices to demonstrate that the following relation
\[
\underset{m_{1},m_{2}\rightarrow\infty}{\lim}\Re_{m_{1},m_{2},\eta_{1}%
,\eta_{2},\gamma_{1},\gamma_{2}}^{(\alpha_{1},\alpha_{2,},s_{1},s_{2}%
)}(e_{u,v};z,y)=e_{u,v},
\]
 uniformly on $I^{2}.$ Then, by Lemma \ref{L3} we may write
\[
\underset{m_{1},m_{2}\rightarrow\infty}{\lim}\Re_{m_{1},m_{2},\eta_{1}%
,\eta_{2},\gamma_{1},\gamma_{2}}^{(\alpha_{1},\alpha_{2,},s_{1},s_{2}%
)}(e_{0,0};z,y)=e_{0,0},
\]%
\[
\underset{m_{1},m_{2}\rightarrow\infty}{\lim}\Re_{m_{1},m_{2},\eta_{1}%
,\eta_{2},\gamma_{1},\gamma_{2}}^{(\alpha_{1},\alpha_{2,},s_{1},s_{2}%
)}(e_{1,0};z,y)=e_{1,0},\underset{m_{1},m_{2}\rightarrow\infty}{\lim}%
\Re_{m_{1},m_{2},\eta_{1},\eta_{2},\gamma_{1},\gamma_{2}}^{(\alpha_{1}%
,\alpha_{2,},s_{1},s_{2})}(e_{0,1};z,y)=e_{0,1},
\]
and
\[
\underset{m_{1},m_{2}\rightarrow\infty}{\lim}\Re_{m_{1},m_{2},\eta_{1}%
,\eta_{2},\gamma_{1},\gamma_{2}}^{(\alpha_{1},\alpha_{2,},s_{1},s_{2}%
)}(e_{2,0}+e_{0,2};z,y)=e_{2,0}+e_{0,2}.
\]
According to the Korovkin theorems presented by Volkov \cite{Vol}, we have thus
\[
\underset{m_{1},m_{2}\rightarrow\infty}{\lim}\Re_{m_{1},m_{2},\eta_{1}%
,\eta_{2},\gamma_{1},\gamma_{2}}^{(\alpha_{1},\alpha_{2,},s_{1},s_{2})}%
(\phi;z,y)=\phi,
\]
converges uniformly.
\end{proof}

Additionally, we will determine the degree of convergence of the operators (\ref{e15}) by using complete and partial modulus of continuity.

For $\phi(z,y)\in C(I^{2})$,the complete modulus of continuity given by
\[
\omega(\phi,\rho)=\sup\left\{  \left\vert \phi(v_{1},v_{2})-\phi
(z,y)\right\vert :\sqrt{(v_{1}-z)^{2}+(v_{2}-y)^{2}}\leq\rho\right\}  ,
\]
for all $(v_{1},v_{2}),(z,y)\in I^{2}.$ Also, the partial modulus continuity
of $\phi(z,y)$ defined by
\begin{align*}
\omega_{1}(\phi,\rho_{1})  &  =\sup\left\{  \left\vert \phi(u_{1}%
,z)-\phi(u_{2},z)\right\vert :z\in\lbrack0,1]\text{ and }\left\vert
u_{1}-u_{2}\right\vert \leq\rho_{1}\right\}  ,\\
\omega_{2}(\phi,\rho_{2})  &  =\sup\left\{  \left\vert \phi(y,v_{1}%
)-\phi(y,v_{2})\right\vert :y\in\lbrack0,1]\text{ and }\left\vert v_{1}%
-v_{2}\right\vert \leq\rho_{2}\right\}  .
\end{align*}

\begin{theorem}
\label{t1} For the operators $\Re_{m_{1},m_{2},\eta_{1},\eta_{2},\gamma
_{1},\gamma_{2}}^{(\alpha_{1},\alpha_{2,},s_{1},s_{2})}(\phi;z,y)$ the
following inequality
\[
\left\vert \Re_{m_{1},m_{2},\eta_{1},\eta_{2},\gamma_{1},\gamma_{2}}%
^{(\alpha_{1},\alpha_{2,},s_{1},s_{2})}(\phi;z,y)-\phi(z,y)\right\vert
\leq4\omega(\phi,\sqrt{\varphi_{m_{1},\eta_{1},\gamma_{1}}^{(\alpha_{1}%
,s_{1})}(z)},\sqrt{\varphi_{m_{2},\eta_{2},\gamma_{2}}^{(\alpha_{2},s_{2}%
)}(y)}),
\]
holds, where $\varphi_{m_{1},\eta_{1},\gamma_{1}}^{(\alpha_{1},s_{1})}(z)$ and
$\varphi_{m_{2},\eta_{2},\gamma_{2}}^{(\alpha_{2},s_{2})}(y)$ are same as in
Corollary \ref{C1}
\end{theorem}

\begin{proof}
Using the linearity property of the operators (\ref{e15}) and the definition of the complete modulus of continuity, we can express it as follows:
\begin{align*}
&  \left\vert \Re_{m_{1},m_{2},\eta_{1},\eta_{2},\gamma_{1},\gamma_{2}%
}^{(\alpha_{1},\alpha_{2,},s_{1},s_{2})}(\phi;z,y)-\phi(z,y)\right\vert \\
&  =\left\vert \Re_{m_{1},m_{2},\eta_{1},\eta_{2},\gamma_{1},\gamma_{2}%
}^{(\alpha_{1},\alpha_{2,},s_{1},s_{2})}(\phi(s,t);z,y)-R_{m,n}^{\alpha
_{1},\alpha_{2,}\beta_{1},\beta_{2,}\gamma_{1},\gamma_{2}}(\phi
(y,z);z,y)\right\vert \\
&  \leq\Re_{m_{1},m_{2},\eta_{1},\eta_{2},\gamma_{1},\gamma_{2}}^{(\alpha
_{1},\alpha_{2,},s_{1},s_{2})}(\left\vert \phi(s,t)-\phi(z,y)\right\vert
;z,y)\\
&  \leq\Re_{m_{1},m_{2},\eta_{1},\eta_{2},\gamma_{1},\gamma_{2}}^{(\alpha
_{1},\alpha_{2,},s_{1},s_{2})}(\omega(\sqrt{(t-z)^{2}+(s-y)^{2}};z,y))\\
&  \leq\omega(\phi,\sqrt{\varphi_{m_{1},\eta_{1},\gamma_{1}}^{(\alpha
_{1},s_{1})}(z)},\sqrt{\varphi_{m_{2},\eta_{2},\gamma_{2}}^{(\alpha_{2}%
,s_{2})}(y)})\\
&  \times\left[  1+\frac{1}{\sqrt{\varphi_{m_{1},\eta_{1},\gamma_{1}}%
^{(\alpha_{1},s_{1})}(z)\varphi_{m_{2},\eta_{2},\gamma_{2}}^{(\alpha_{2}%
,s_{2})}(y)}}\Re_{m_{1},m_{2},\eta_{1},\eta_{2},\gamma_{1},\gamma_{2}%
}^{(\alpha_{1},\alpha_{2,},s_{1},s_{2})}(\sqrt{(t-z)^{2}+(s-y)^{2}%
};z,y)\right]  .
\end{align*}
Implementing the Cauchy-Schwarz inequality, then
\begin{align*}
&  \left\vert \Re_{m_{1},m_{2},\eta_{1},\eta_{2},\gamma_{1},\gamma_{2}%
}^{(\alpha_{1},\alpha_{2,},s_{1},s_{2})}(\phi;z,y)-\phi(z,y)\right\vert \\
&  \leq\omega(\phi,\sqrt{\varphi_{m_{1},\eta_{1},\gamma_{1}}^{(\alpha
_{1},s_{1})}(z)},\sqrt{\varphi_{m_{2},\eta_{2},\gamma_{2}}^{(\alpha_{2}%
,s_{2})}(y)})\\
&  \times\left[  1+\frac{\sqrt{\Re_{m_{1},m_{2},\eta_{1},\eta_{2},\gamma
_{1},\gamma_{2}}^{(\alpha_{1},\alpha_{2,},s_{1},s_{2})}((t-z)^{2}%
;z,y)\Re_{m_{1},m_{2},\eta_{1},\eta_{2},\gamma_{1},\gamma_{2}}^{(\alpha
_{1},\alpha_{2,},s_{1},s_{2})}((s-y)^{2};z,y)}}{\sqrt{\varphi_{m_{1},\eta
_{1},\gamma_{1}}^{(\alpha_{1},s_{1})}(z)\varphi_{m_{2},\eta_{2},\gamma_{2}%
}^{(\alpha_{2},s_{2})}(y)}}\right. \\
&  +\left.  \frac{\sqrt{\Re_{m_{1},m_{2},\eta_{1},\eta_{2},\gamma_{1}%
,\gamma_{2}}^{(\alpha_{1},\alpha_{2,},s_{1},s_{2})}((t-z)^{2};z,y)}}%
{\sqrt{\varphi_{m_{1},\eta_{1},\gamma_{1}}^{(\alpha_{1},s_{1})}(z)}}%
+\frac{\sqrt{\Re_{m_{1},m_{2},\eta_{1},\eta_{2},\gamma_{1},\gamma_{2}%
}^{(\alpha_{1},\alpha_{2,},s_{1},s_{2})}((s-y)^{2};z,y)}}{\sqrt{\varphi
_{m_{2},\eta_{2},\gamma_{2}}^{(\alpha_{2},s_{2})}(y)}}\right]  .
\end{align*}
Taking $\varphi_{m_{1},\eta_{1},\gamma_{1}}^{(\alpha_{1},s_{1})}(z)=\Re
_{m_{1},m_{2},\eta_{1},\eta_{2},\gamma_{1},\gamma_{2}}^{(\alpha_{1}%
,\alpha_{2,},s_{1},s_{2})}((t-z)^{2};z,y)$ and $\varphi_{m_{2},\eta_{2}%
,\gamma_{2}}^{(\alpha_{2},s_{2})}(y)=\Re_{m_{1},m_{2},\eta_{1},\eta_{2}%
,\gamma_{1},\gamma_{2}}^{(\alpha_{1},\alpha_{2,},s_{1},s_{2})}((s-y)^{2}%
;z,y),$ this gives the assertion.
\end{proof}

\begin{theorem}
For the operators $\Re_{m_{1},m_{2},\eta_{1},\eta_{2},\gamma_{1},\gamma_{2}%
}^{(\alpha_{1},\alpha_{2,},s_{1},s_{2})}(\phi;y,z)$ the following inequality
\[
\left\vert \Re_{m_{1},m_{2},\eta_{1},\eta_{2},\gamma_{1},\gamma_{2}}%
^{(\alpha_{1},\alpha_{2,},s_{1},s_{2})}(\phi;z,y)-\phi(z,y)\right\vert
\leq2\left(  \omega_{1}(\phi,\sqrt{\varphi_{m_{1},\eta_{1},\gamma_{1}%
}^{(\alpha_{1},s_{1})}(z)})+\omega_{2}(\phi,\sqrt{\varphi_{m_{2},\eta
_{2},\gamma_{2}}^{(\alpha_{2},s_{2})}(y)})\right)  ,
\]
holds, where $\varphi_{m_{1},\eta_{1},\gamma_{1}}^{(\alpha_{1},s_{1})}(z)$ and
$\varphi_{m_{2},\eta_{2},\gamma_{2}}^{(\alpha_{2},s_{2})}(y)$ are defined in previous
Theorem.
\end{theorem}

\begin{proof}
Based on the definition of the partial modulus of continuity and applying the Cauchy-Schwarz inequality, we can express it as follows:
\begin{align*}
&  \left\vert \Re_{m_{1},m_{2},\eta_{1},\eta_{2},\gamma_{1},\gamma_{2}%
}^{(\alpha_{1},\alpha_{2,},s_{1},s_{2})}(\phi;z,y)-\phi(z,y)\right\vert \\
&  =\left\vert \Re_{m_{1},m_{2},\eta_{1},\eta_{2},\gamma_{1},\gamma_{2}%
}^{(\alpha_{1},\alpha_{2,},s_{1},s_{2})}(\phi(s,t);z,y)-\Re_{m_{1},m_{2}%
,\eta_{1},\eta_{2},\gamma_{1},\gamma_{2}}^{(\alpha_{1},\alpha_{2,},s_{1}%
,s_{2})}(\phi(z,y);z,y)\right\vert \\
&  \leq\Re_{m_{1},m_{2},\eta_{1},\eta_{2},\gamma_{1},\gamma_{2}}^{(\alpha
_{1},\alpha_{2,},s_{1},s_{2})}(\left\vert \phi(s,t)-\phi(z,y)\right\vert
;z,y)\\
&  \leq\Re_{m_{1},m_{2},\eta_{1},\eta_{2},\gamma_{1},\gamma_{2}}^{(\alpha
_{1},\alpha_{2,},s_{1},s_{2})}(\left\vert \phi(s,t)-\phi(z,t)\right\vert
;z,y)+\Re_{m_{1},m_{2},\eta_{1},\eta_{2},\gamma_{1},\gamma_{2}}^{(\alpha
_{1},\alpha_{2,},s_{1},s_{2})}(\left\vert \phi(z,t)-\phi(z,y)\right\vert
;z,y)\\
&  \leq\omega_{1}(\phi,\varphi_{m_{1},\eta_{1},\gamma_{1}}^{(\alpha_{1}%
,s_{1})}(z))\left[  1+\frac{1}{\varphi_{m_{1},\eta_{1},\gamma_{1}}%
^{(\alpha_{1},s_{1})}(z)}\Re_{m_{1},m_{2},\eta_{1},\eta_{2},\gamma_{1}%
,\gamma_{2}}^{(\alpha_{1},\alpha_{2,},s_{1},s_{2})}(\left\vert s-z\right\vert
;z,y)\right]  \\
&  +\omega_{2}(\phi,\varphi_{m_{2},\eta_{2},\gamma_{2}}^{(\alpha_{2},s_{2}%
)}(y))\left[  1+\frac{1}{\varphi_{m_{2},\eta_{2},\gamma_{2}}^{(\alpha
_{2},s_{2})}(y)}\Re_{m_{1},m_{2},\eta_{1},\eta_{2},\gamma_{1},\gamma_{2}%
}^{(\alpha_{1},\alpha_{2,},s_{1},s_{2})}(\left\vert t-y\right\vert
;z,y)\right]  .
\end{align*}
Consequently, we have
\begin{align*}
&  \left\vert \Re_{m_{1},m_{2},\eta_{1},\eta_{2},\gamma_{1},\gamma_{2}%
}^{(\alpha_{1},\alpha_{2,},s_{1},s_{2})}(\phi;z,y)-\phi(z,y)\right\vert
\leq\omega_{1}(\phi,\varphi_{m_{1},\eta_{1},\gamma_{1}}^{(\alpha_{1},s_{1}%
)}(z))\left[  1+\frac{\sqrt{\Re_{m_{1},m_{2},\eta_{1},\eta_{2},\gamma
_{1},\gamma_{2}}^{(\alpha_{1},\alpha_{2,},s_{1},s_{2})}((t-z)^{2};z,y)}%
}{\varphi_{m_{1},\eta_{1},\gamma_{1}}^{(\alpha_{1},s_{1})}(z)}\right]  \\
&  +\omega_{2}(\phi,\varphi_{m_{2},\eta_{2},\gamma_{2}}^{(\alpha_{2},s_{2}%
)}(y))\left[  1+\frac{\sqrt{\Re_{m_{1},m_{2},\eta_{1},\eta_{2},\gamma
_{1},\gamma_{2}}^{(\alpha_{1},\alpha_{2,},s_{1},s_{2})}((s-y)^{2};z,y)}%
}{\varphi_{m_{2},\eta_{2},\gamma_{2}}^{(\alpha_{2},s_{2})}(y)}\right]  .
\end{align*}
Hence, choosing $\varphi_{m_{1},\eta_{1},\gamma_{1}}^{(\alpha_{1},s_{1})}(z)$
and $\varphi_{m_{2},\eta_{2},\gamma_{2}}^{(\alpha_{2},s_{2})}(y)$ as in
Theorem \ref{t1}, we get the desired assertion.
\end{proof}
Further, we will provide various graphical representations and numerical error of approximation tables of the operators $\Re_{m_{1},m_{2},\eta_{1},\eta_{2},\gamma
	_{1},\gamma_{2}}^{(\alpha_{1},\alpha_{2,},s_{1},s_{2})}(\phi;z,y)$.
	
	\begin{figure}[h]
		\centering
		\includegraphics[width=11cm, height=9cm]{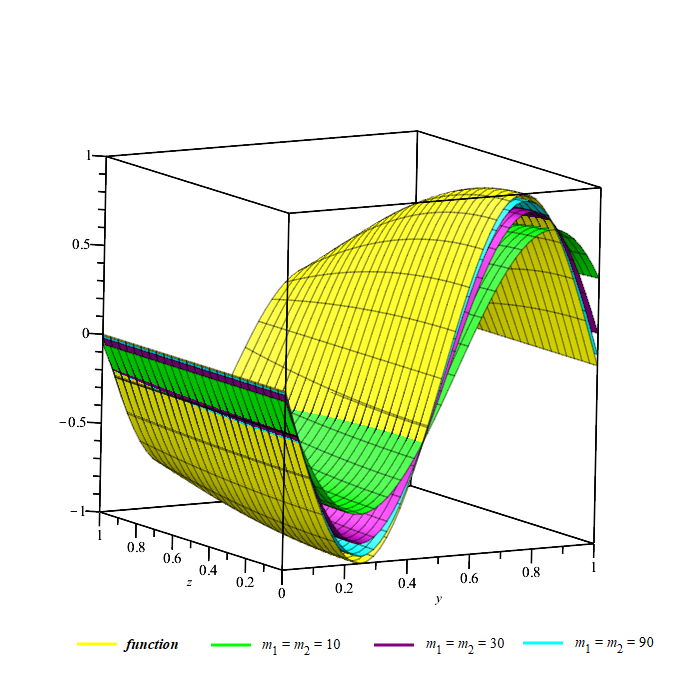} \caption{Convergence of operators
			$\Re_{m_{1},m_{2},2,2,3,3}^{(0.9,0.9,2,2)}(\phi;z,y)$ to
			$\phi(;z,y)=(yz^2-1)\sin(2\pi y)$ for $m_{1}=m_{2}=10,30,90$ }%
		\label{fig4}%
	\end{figure}
\begin{table}[h]
	\caption{The error of approximation of operators $\Re_{m_{1},m_{2},2,2,3,3}^{(0.9,0.9,2,2)}(\phi;z,y)$ to
		$\phi(;z,y)=(yz^2-1)\sin(2\pi y)$ for $m_{1}=m_{2}=10,30,90$}%
	\label{tab5}
	\rowcolors{0.5}{anti-flashwhite}{gray(x11gray)}
	\begin{tabular}
		[c]{c|c|c|c|c}\hline
		$(z,y)$&$|\Re_{10,10,2,2,3,3%
		}^{(0.9,0.9,2,2)}(\phi;z,y)-\phi(z,y)|$& $|\Re_{30,30,2,2,3,3%
		}^{(0.9,0.9,2,2)}(\phi;z,y)-\phi(z,y)|$& $|\Re_{90,90,2,2,3,3%
		}^{(0.9,0.9,2,2)}(\phi;z,y)-\phi(z,y)|$& \\\hline
		(0.1,0.1)& 0.100371117& 0.034798399& 0.011680216& \\
		(0.2,0.2)& 0.248811999& 0.097459204& 0.034498155& \\
		(0.3,0.3)& 0.251918792& 0.101864734& 0.036610920& \\
		(0.4,0.4)& 0.084553405& 0.030420691& 0.010448410& \\
		(0.5,0.5)& 0.152540436& 0.071422784& 0.026978653& \\
		(0.6,0.6)& 0.297414896& 0.127200247& 0.046443579& \\
		(0.7,0.7)& 0.241080501& 0.090851633& 0.031299415& \\
		(0.8,0.8)& 0.022308261& 0.006763693& 0.004701082& \\
		(0.9,0.9)& 0.169363792& 0.069756801& 0.024705205& \\\hline
	\end{tabular}
\end{table}
	In Figure
\ref{fig4} and Table \ref{tab5}, we
examine the function
\[
\phi(;z,y)=(yz^2-1)\sin(2\pi y)
\]
on the interval $I^{2}$. In Figure
\ref{fig4}, we present the convergence of operators $\Re_{m_{1},m_{2},\eta_{1},\eta_{2},\gamma
	_{1},\gamma_{2}}^{(\alpha_{1},\alpha_{2,},s_{1},s_{2})}(\phi;z,y)$ to $\phi(;z,y)=(yz^2-1)\sin(2\pi y)$ for fixed parameters $\eta_{1}=\eta_{2}=2,\gamma
_{1}=\gamma
_{2}=3, s_{1}=s_{2}=2, \alpha_{1}=\alpha_{2}=0.9$ while $m_{1}=m_{2}=10,30,90$ values increase respectively. Also, in Table \ref{tab5}, for the same parameter values we compare the maximum error of approximation of operators $\Re_{m_{1},m_{2},\eta_{1},\eta_{2},\gamma
	_{1},\gamma_{2}}^{(\alpha_{1},\alpha_{2,},s_{1},s_{2})}(\phi;z,y)$ to $\phi(;z,y)=(yz^2-1)\sin(2\pi y)$. It can be easily seen by Figure \ref{fig4} and Table \ref{tab5} that
when the values $m_{1},m_{2}$ increase then the approximation of
operators $\Re_{m_{1},m_{2},\eta_{1},\eta_{2},\gamma
	_{1},\gamma_{2}}^{(\alpha_{1},\alpha_{2,},s_{1},s_{2})}(\phi;z,y)$ to $\phi(;z,y)=(yz^2-1)\sin(2\pi y)$ turn into better for different values $(z,y)\in I^{2}$. 
	
	\begin{figure}[h]
		\centering
		\includegraphics[width=11cm, height=9cm]{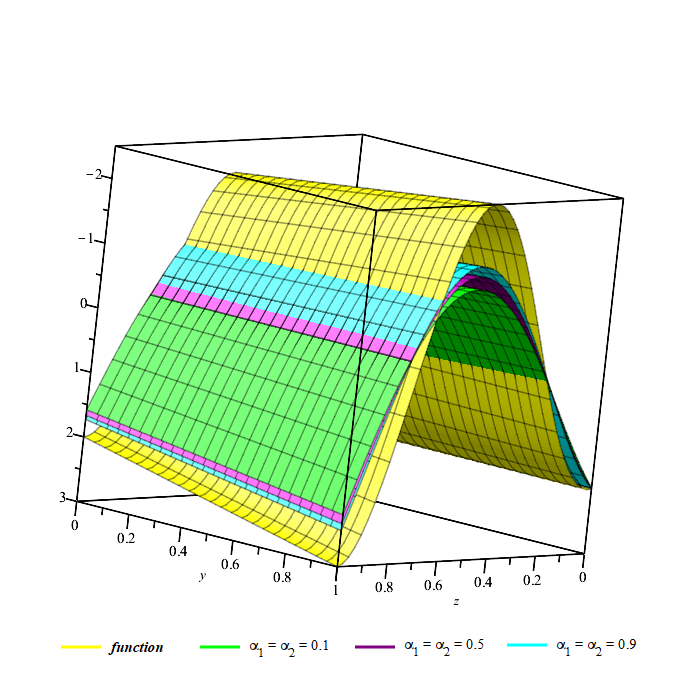} \caption{Convergence of operators
			$\Re_{15,15,2,2,2,2}^{(\alpha_{1},\alpha_{2},2,2)}(\phi;z,y)$ to
			$\phi(;z,y)=(yz+2)\cos(2\pi z)$ for $\alpha_{1}=\alpha_{2}=0.1,0.5,0.9$ }%
		\label{fig5}%
	\end{figure}
\begin{table}[h]
	\caption{The error of approximation of operators
		$\Re_{15,15,2,2,2,2}^{(\alpha_{1},\alpha_{2},2,2)}(\phi;z,y)$ to
		$\phi(;z,y)=(yz+2)\cos(2\pi z)$ for $\alpha_{1}=\alpha_{2}=0.1,0.5,0.9$}%
	\label{tab6}
	\rowcolors{0.5}{anti-flashwhite}{gray(x11gray)}
	\begin{tabular}
		[c]{c|c|c|c|c}\hline
		$(z,y)$&$|\Re_{15,15,2,2,2,2%
		}^{(0.1,0.1,2,2)}(\phi;z,y)-\phi(z,y)|$& $|\Re_{15,15,2,2,2,2%
		}^{(0.5,0.5,2,2)}(\phi;z,y)-\phi(z,y)|$& $|\Re_{15,15,2,2,2,2%
		}^{(0.9,0.9,2,2)}(\phi;z,y)-\phi(z,y)|$& \\\hline
		(0.1,0.1)& 0.195457041& 0.189553367& 0.183649692& \\
		(0.2,0.2)& 0.078242461& 0.076050756& 0.073859051& \\
		(0.3,0.3)& 0.245309369& 0.237433313& 0.229557258& \\
		(0.4,0.4)& 0.563370466& 0.544366814& 0.525363162& \\
		(0.5,0.5)& 0.656509598& 0.630609173& 0.604708748& \\
		(0.6,0.6)& 0.413054142& 0.387853505& 0.362652867& \\
		(0.7,0.7)& 0.125938158& 0.109033855& 0.092129552& \\
		(0.8,0.8)& 0.605991868& 0.601044280& 0.596096692& \\
		(0.9,0.9)& 0.743564319& 0.740101376& 0.736638433& \\\hline
	\end{tabular}
\end{table}
	In Figure
\ref{fig5} and Table \ref{tab6}, we
examine the function
\[
\phi(;z,y)=(yz+2)\cos(2\pi z)
\]
on the interval $I^{2}$. In Figure
\ref{fig5}, we present the convergence behavior of operators $\Re_{m_{1},m_{2},\eta_{1},\eta_{2},\gamma
	_{1},\gamma_{2}}^{(\alpha_{1},\alpha_{2,},s_{1},s_{2})}(\phi;z,y)$ to $\phi(;z,y)=(yz+2)\cos(2\pi z)$ for fixed parameters $\eta_{1}=\eta_{2}=2,\gamma
_{1}=\gamma
_{2}=2, s_{1}=s_{2}=2, m_{1}=m_{2}=15$ while $\alpha_{1}=\alpha_{2}=0.1,0.5,0.9$ values increase respectively. Also, in Table \ref{tab6}, for the same parameter values we compare the maximum error of approximation of operators $\Re_{m_{1},m_{2},\eta_{1},\eta_{2},\gamma
	_{1},\gamma_{2}}^{(\alpha_{1},\alpha_{2,},s_{1},s_{2})}(\phi;z,y)$ to $\phi(;z,y)=(yz^2-1)\sin(2\pi y)$. One can check from Figure \ref{fig5} and Table \ref{tab6} that,
while the values $\alpha_{1},\alpha_{2}$ approaches to 1 then the approximation of
operators $\Re_{m_{1},m_{2},\eta_{1},\eta_{2},\gamma
	_{1},\gamma_{2}}^{(\alpha_{1},\alpha_{2,},s_{1},s_{2})}(\phi;z,y)$ to $\phi(;z,y)=(yz+2)\cos(2\pi z)$ present superior convergence
behavior on $(z,y)\in I^{2}$.

\begin{figure}[h]
	\centering
	\includegraphics[width=11cm, height=9cm]{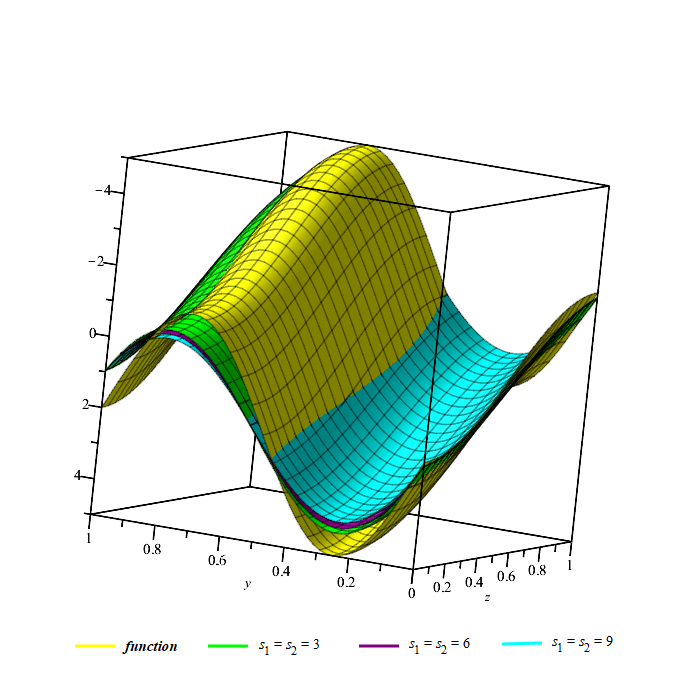} \caption{Convergence of operators
		$\Re_{15,15,3,3,2,2}^{(0.8,0.8,s_{1},s_{2})}(\phi;z,y)$ to
		$\phi(;z,y)=2\cos(\pi z)+3\sin(2\pi y)$ for $s_{1}=s_{2}=3,6,9$ }%
	\label{fig6}%
\end{figure}

\begin{table}[h]
	\caption{The error of approximation of operators $\Re_{15,15,3,3,2,2}^{(0.8,0.8,s_{1},s_{2})}(\phi;z,y)$ to
		$\phi(;z,y)=2\cos(\pi z)+3\sin(2\pi y)$ for $s_{1}=s_{2}=3,6,9$}%
	\label{tab7}
	\rowcolors{0.5}{anti-flashwhite}{gray(x11gray)}
	\begin{tabular}
		[c]{c|c|c|c|c}\hline
		$(z,y)$&$|\Re_{15,15,3,3,2,2%
		}^{(0.8,0.8,9,9)}(\phi;z,y)-\phi(z,y)|$& $|\Re_{15,15,3,3,2,2%
		}^{(0.8,0.8,6,6)}(\phi;z,y)-\phi(z,y)|$& $|\Re_{15,15,3,3,2,2%
		}^{(0.8,0.8,3,3)}(\phi;z,y)-\phi(z,y)|$& \\\hline
		(0.1,0.1)& 0.489005750& 0.363729543& 0.270971583& \\
		(0.2,0.2)& 0.960387677& 0.773290789& 0.622200863& \\
		(0.3,0.3)& 0.962695442& 0.776177886& 0.618766351& \\
		(0.4,0.4)& 0.432043739& 0.297004352& 0.183925249& \\
		(0.5,0.5)& 0.487170570& 0.452180495& 0.400125636& \\
		(0.6,0.6)& 1.123225970& 1.085080625& 1.034878809& \\
		(0.7,0.7)& 1.350081013& 1.240427722& 1.127066981& \\
		(0.8,0.8)& 0.915557368& 0.777200727& 0.646573082& \\
		(0.9,0.9)& 0.214512851& 0.123898803& 0.018575111& \\\hline
	\end{tabular}
\end{table}
	In Figure
\ref{fig6} and Table \ref{tab7}, we
examine the function
\[
\phi(;z,y)=2\cos(\pi z)+3\sin(2\pi y)
\]
on the interval $I^{2}$. In Figure
\ref{fig6}, we show convergence behavior of operators $\Re_{m_{1},m_{2},\eta_{1},\eta_{2},\gamma
	_{1},\gamma_{2}}^{(\alpha_{1},\alpha_{2,},s_{1},s_{2})}(\phi;z,y)$ to $\phi(;z,y)=2\cos(\pi z)+3\sin(2\pi y)$ for fixed parameters $\eta_{1}=\eta_{2}=3,\gamma
_{1}=\gamma
_{2}=2, \alpha_{1}=\alpha_{2}=0.8, m_{1}=m_{2}=15$ while $s_{1}=s_{2}=3,6,9$ values increase respectively. Also, in Table \ref{tab7}, for the same parameter values we compare the maximum error of approximation of operators $\Re_{m_{1},m_{2},\eta_{1},\eta_{2},\gamma
	_{1},\gamma_{2}}^{(\alpha_{1},\alpha_{2,},s_{1},s_{2})}(\phi;z,y)$ to $\phi(;z,y)=2\cos(\pi z)+3\sin(2\pi y)$. One can easily check from Figure \ref{fig6} and Table \ref{tab7} that, for small values of $s_{1},s_{2}$ then approximation of
operators $\Re_{m_{1},m_{2},\eta_{1},\eta_{2},\gamma
	_{1},\gamma_{2}}^{(\alpha_{1},\alpha_{2,},s_{1},s_{2})}(\phi;z,y)$ to $\phi(;z,y)=2\cos(\pi z)+3\sin(2\pi y)$ becomes better  on $(z,y)\in I^{2}$.

\section{Conclusion}

In this study, we proposed Riemann-Liouville type fractional a novel generalization of Bernstein-Kantorovich version operators. Some basic moment estimates of operators $\Re_{m,\eta,\gamma}^{(\alpha
	,s)}(\phi;z)$ are investigated. Subsequently, we delved into the exploration of some local
direct approximation theorems related with these new operators. Additionally, we illustrated various graphical and numerical
examples to demonstrate the convergence behavior, computational efficiency, advantage,
and consistency of operators $\Re_{m,\eta,\gamma}^{(\alpha
	,s)}(\phi;z)$ under various parameter configurations. Moreover, this new operator was compared with other operators and better approximation results were obtained in terms of some parameters.  Also, we considered the bivariate extension of operators $\Re_{m,\eta,\gamma}^{(\alpha
	,s)}(\phi;z)$ 
and studied degree of approximation in terms of partial
and complete modulus of continuity. Finally, we provided some graphical
representations and absolute error of approximation results to verify the convergence
behavior of bivariate form of proposed operators under different parameter configurations.

\noindent\textbf{Declarations}\newline

\noindent\textbf{Funding:} The author does not receive support from any
organization for the submitted work.\newline

\noindent\textbf{Conflict of interest:} The author declares that there is no
known competing financial interests or personal relationship that could have
appeared to influence the work reported in this article.\newline

\noindent\textbf{Data Availability Statement:} All data generated or analyzed
during this study are included in this published article.

\end{document}